\documentclass[12pt,reqno]{amsart}

\usepackage[utf8]{inputenc}
\usepackage{amsmath}
\usepackage{amsfonts}
\usepackage{xypic}
\usepackage{amsthm}
\usepackage{mathrsfs}
\usepackage{amssymb}
\usepackage[all]{xy}
\usepackage[T1]{fontenc}
\usepackage{graphicx}
\usepackage{multirow}
\usepackage{hyperref}

\numberwithin{equation}{section}

\newtheorem{theorem}{Theorem}[section]
\newtheorem{lemma}[theorem]{Lemma}
\newtheorem{proposition}[theorem]{Proposition}
\newtheorem{corollary}[theorem]{Corollary}

\theoremstyle{remark}
\newtheorem{remark}[theorem]{Remark}

\newcommand{\conj}{\mathrm{conj}}

\newcommand{\Pic}{\mathrm{Pic}}

\newcommand{\Hom}{\mathrm{Hom}}
\newcommand{\Aut}{\mathrm{Aut}}
\newcommand{\Int}{\mathrm{Int}}
\newcommand{\Conj}{\mathrm{Conj}}
\newcommand{\Out}{\mathrm{Out}}

\newcommand{\ad}{\mathrm{ad}}
\newcommand{\id}{\mathrm{Id}}
\newcommand{\GL}{\mathrm{GL}}
\newcommand{\Sing}{\mathrm{Sing}}
\newcommand{\U}{\mathrm{U}}
\newcommand{\ord}{\mathrm{ord}}
\newcommand{\aj}{\mathrm{aj}}

\newcommand{\smooth}{sm}
\newcommand{\Sym}{Sym}

\renewcommand{\i}{\mathbf{i}}

\newcommand{\Ff}{\mathcal{F}}
\newcommand{\Mm}{\mathcal{M}}
\newcommand{\Oo}{\mathcal{O}}
\newcommand{\Rr}{\mathcal{R}}

\newcommand{\PP}{\mathbb{P}}

\newcommand{\CC}{\mathbb{C}}
\newcommand{\HH}{\mathbb{H}}
\newcommand{\RR}{\mathbb{R}}
\newcommand{\ZZ}{\mathbb{Z}}

\newcommand{\gG}{\mathfrak{g}}
\newcommand{\tT}{\mathfrak{t}}
\newcommand{\zZ}{\mathfrak{z}}

\DeclareMathOperator{\SL}{SL}

\begin{document}

\title[Involutions of the moduli of Higgs bundles over elliptic curves]{Involutions 
of the moduli spaces of $G$-Higgs bundles over elliptic curves}

\author[I. Biswas]{Indranil Biswas}

\address{Indranil Biswas \\ School of Mathematics, Tata Institute of Fundamental
Research, \\ Homi Bhabha Road, Mumbai 400005, India}

\email{indranil@math.tifr.res.in}

\author[L. A. Calvo]{Luis Angel Calvo}

\address{Luis Angel Calvo \\ ICMAT (Instituto de Ciencias Matem\'aticas) \\ C/ Nicol\'as
Cabrera, no. 13--15, Campus Cantoblanco, 28049 Madrid, Spain}

\email{luis.calvo@icmat.es}

\author[E. Franco]{Emilio Franco}

\address{Emilio Franco \\ IMECC (Instituto de Matem\'atica, Estat\'istica e Computa\c{c}\~ao Cien\-t\'i\-fi\-ca), Universidade Estadual de Campinas \\ 
Rua S\'ergio Buarque de Holanda 651, Cidade Universit\'aria "Zeferino Vaz", Campinas (SP, Brazil)}

\email{emilio\_franco@ime.unicamp.br}

\author[O. Garc\'{\i}a-Prada]{Oscar Garc\'{\i}a-Prada}

\address{Oscar Garc\'ia-Prada \\ ICMAT (Instituto de Ciencias Matem\'aticas) \\ C/ Nicol\'as
Cabrera, no. 13--15, Campus Cantoblanco, 28049 Madrid, Spain}

\email{oscar.garcia-prada@icmat.es}



\thanks{
This work is supported by the European Commission Marie Curie IRSES  MODULI Programme PIRSES-GA-2013-612534. 
The first author is supported by a J. C. Bose Fellowship. The third 
author is supported by Funda\c{c}\~{a}o de Amparo \`a Pesquisa do Estado 
de S\~{a}o Paulo through projects FAPESP-2012/16356-6 and BEPE-2015/06696-2. 
The fourth author was partially supported by the Ministerio de Econom\'{\i}a y 
Competitividad of Spain through Project MTM2013-43963-P and Severo Ochoa 
Excellence Grant. 
  }

\begin{abstract}
We present a systematic study of involutions on the moduli space of $G$-Higgs bundles 
over an elliptic curve $X$, where $G$ is complex reductive
affine algebraic group. The 
fixed point loci in the moduli space of $G$-Higgs bundles on $X$, and in the moduli space of 
representations of the fundamental group of $X$ into $G$, are described. This leads to an
explicit description of the moduli spaces of pseudo-real $G$-Higgs bundles over $X$.
\end{abstract}

\maketitle

\tableofcontents

\section{Introduction}

Given a complex reductive affine algebraic group $G$ and a compact Riemann surface $X$, a
{\it $G$-Higgs bundle} over $X$ is a pair $(E,\,\varphi)$, where $E$ is a holomorphic 
principal $G$-bundle over $X$, and $\varphi$ is a holomorphic section of 
$E(\mathfrak{g})\otimes K$, the adjoint vector bundle $E(\mathfrak{g})$
of $E$ twisted by the canonical bundle 
$K$ of $X$. These objects were introduced by Hitchin \cite{hitchin-selfuality_equations},
while Simpson \cite{simpson0, simpson1, simpson2}, and Nitsure \cite{Ni} constructed the moduli space
$\Mm(G)$ of $G$-Higgs bundles.

Let $X$ be a complex elliptic curve.
Let $\sigma_{+}$ and $\alpha_{+}$ denote holomorphic involutions on $G$ and $X$, 
respectively. Consider the involution (see Section \ref{sc brane involutions} for 
details)
\[
I(\alpha_{+},\sigma_{+},\pm)\, :\, \Mm(G) \,\longrightarrow\,
\Mm(G)\, , \ \
(E,\,\varphi)\, \longmapsto\,
( \alpha_{+}^*\sigma_{+}(E),\, \pm \alpha_{+}^*\sigma_{+}(\varphi)) \, .
\]
Analogously, for anti-holomorphic involutions $\sigma_{-}$ and $\alpha_{-}$, on $G$ and $X$, define
\[
I(\alpha_{-},\sigma_{-},\pm)\, :\, 
\Mm(G) \,\longrightarrow\,
\Mm(G) \, , \ \ (E,\,\varphi) \, \longmapsto\,
( \alpha_{-}^*\sigma_{-}(E), \,\pm\alpha_{-}^*\sigma_{-}(\varphi))\, .
\]
The first goal of this paper is to classify these involutions and describe their
fixed points. 

Let $Z$ denote the center of $G$ and $Z_{2}^{\sigma_{-}}\, \subset\, Z^{\sigma_{-}}$
the group elements of order two of $Z$ fixed pointwise by $\sigma_{-}$.
Given any $z \,\in\, Z_{2}^{\sigma_{-}}$, we 
consider the pseudo-real $(G,\alpha_{-},\sigma_{-},\pm,z)$-Higgs bundles \cite{BGP, BGH} 
which are $G$-Higgs bundles equipped with a certain real structure (see Section 
\ref{sc moduli spaces of fixed points}). The second goal of this paper is to obtain a 
description of the moduli space $\Mm(G,\alpha_{-},\sigma_{-},\pm,z)$ of pseudo-real 
$G$-Higgs bundles; this is achieved in Section \ref{sc 
moduli of pseudo-real}.

\

A major result of the theory of Higgs bundles is the non-abelian Hodge 
correspondence, proved by Hitchin \cite{hitchin-selfuality_equations} and Do\-naldson 
\cite{donaldson} for $\SL(2,\CC)$, and by Simpson \cite{simpson1, simpson2} and 
Corlette \cite{corlette} for an arbitrary group (see \cite{BG} also for the general case). 
It relates the moduli of $G$-Higgs bundles with the moduli space of representations of 
the fundamental  group
\begin{equation}\label{rg}
\Rr(G) \,:= \,\Hom(\pi_1(X,x_{0}),\, G) /\!\!/ G
\end{equation}
as follows:

\begin{theorem}\label{ths}
There is a natural homeomorphism $\Mm(G)\,\cong\, \Rr(G)\, .$
\end{theorem}

To treat simultaneously the holomorphic and anti-holomorphic cases, we will write
$\sigma_{\epsilon}$ and $\alpha_\epsilon$ (instead of
$\sigma_+ ,\, \sigma_-$ and $\alpha_+ ,\, \alpha_-$).
Along with the study of the involutions $I(\alpha_{\epsilon}, \sigma_{\epsilon}, \pm)$, we also describe in this article their images $J(\alpha_{\epsilon}, \sigma_{\epsilon}, \pm)$ under the homeomorphism of Theorem \ref{ths}, which are involutions on $\Rr(G)$, making the following diagram commutative,
\[
\xymatrix{
\Mm(G) \ar[rr]^{\rm homeo.} \ar[d]_{I(\alpha_{\epsilon}, \sigma_{\epsilon}, \pm)} & & \Rr(G) \ar[d]^{J(\alpha_{\epsilon}, \sigma_{\epsilon}, \pm)}
\\
\Mm(G) \ar[rr]^{\rm homeo.}  & & \Rr(G).
}
\]

The complex structure $\Gamma_1$ on $\Mm(G)$ comes from $X$, while the complex structure $\Gamma_2$ on $\Rr(G)$ is given by that of the group. In view of Theorem \ref{ths}
one can identify the spaces $\Mm(G)$ and $\Rr(G)$. As we have seen,
there are two different complex structures on it and, in fact, on can endow
the smooth locus $\Mm(G)_{sm}$ with a hyper-K\"ahler structure $(\omega_1, \omega_2,
\omega_3)$, where the complex structure $\Gamma_3$ is $\Gamma_1\Gamma_2$. Note also that
one can define three corresponding (holomorphic) symplectic
structures $\Omega_1, \, \Omega_2$ and $\Omega_3$.

In the context of Mirror Symmetry, an $A$-brane is a Lagrangian submanifold together 
with a flat bundle supported on it, while a $B$-brane is a complex submanifold with 
a holomorphic bundle, although in this paper we will use the terms $A$-brane and 
$B$-brane to refer simply to their support, i.e., the special Lagrangian or complex 
submanifold. In their ground-breaking paper \cite{kapustin&witten}, Kapustin and 
Witten introduced the concept of $(*, *, *)$-brane as a submanifold which is either 
complex or special Lagrangian for each of the complex structures $\Gamma_1$, 
$\Gamma_2$ and $\Gamma_3$. This is a hyper-K\"ahler submanifold in the case of a 
$(B,B,B)$-brane, or a $\Omega_i$-complex Lagrangian submanifold in the case of 
$(B,A,A)$, $(A,B,A)$ or $(A,A,B)$-branes.

In \cite{BSC}, Baraglia and Schaposnik identified the fixed points of 
$I(\alpha_{-}, \sigma_{-}, -)$ with an $(A,B,A)$-brane (see also \cite{BGP}). 
The involutions $I(\text{Id}_X, \sigma_{+}, 
+)$ have been studied by Hitchin 
\cite{hitchin-selfuality_equations,hitchin_cc} and in
 \cite{GP1, GP2,garcia-prada&ramanan}. 
It turns out that 
the fixed points of $I(\text{Id}_X, \sigma_{+}, +)$ constitute a $(B,B,B)$-brane, while the 
fixed points of $I(\text{Id}_X, \sigma_{+}, -)$ give a $(B,A,A)$-brane. See 
\cite{garcia-prada&ramanan} for a detailed study of these involutions and their 
fixed points. For a general picture of anti-holomorphic involutions on the 
smooth
locus $\Mm(G)_{\smooth}\, \subset\,\Mm(G)$ we refer to 
\cite{baraglia&schaposnik_2, BGP}. The involution  
$I(\alpha_+, \text{Id}_X, +)$ is a particular case of the more general situation
studied in \cite{garcia-prada&wilkin} (see also \cite{heller&schaposnik}, which contains a detailed study of the case of $\alpha_+$ being fix point free). 

Over an elliptic curve one can achieve a greater level of explicitness
regarding the des\-cription of the moduli spaces of Higgs bundles. In 1957 Atiyah
\cite{atiyah} described the moduli space of rank $n$ vector bundles on an elliptic
curve $X$
as the symmetric product of the curve $M(\GL(n,\CC)) \,\cong\, \Sym^n(X)$ and some 30 years later, Laszlo \cite{laszlo} and Friedman, Morgan and Witten \cite{friedman&morgan, friedman&morgan&witten}, gave a description of the moduli space of principal $G$--bundles for a complex reductive Lie group $G$, in terms of the finite quotient $M(G) \cong (X \otimes_\ZZ \Lambda_T) / W$, where $T$ is a Cartan subgroup of $G$, $\Lambda_T$ is the cocharacter lattice and $W$ the associated Weyl group. In \cite{FGN} the third and fourth authors studied the case of $G$-Higgs bundles, showing
\[
\Mm(G) \,\cong\, (T^*X \otimes_\ZZ \Lambda_T) / W.
\]
One can find in \cite{FGN2} a detailed study of the involution $I(\text{Id}_X, \sigma_{+}, -)$ on the moduli spaces of Higgs bundles for real forms of $G$, and in
\cite{BS}, a description of the fixed points of this involution
$I(\alpha_{-}, \sigma_{-}, \pm)$ of the Atiyah's moduli space $M(\GL(n,\CC))$ and its relation with the moduli spaces of (real) vector bundles over a real elliptic curve (see also \cite{BB} for the case of a Klein bottle).

\

This article is organized as follows. After a quick review of involutions on complex Lie groups in Section \ref{sc involutions}, we define in Section \ref{sc brane involutions} the involutions $I(\alpha_{\epsilon},\sigma_{\epsilon}, \pm)$
which are the objects of our study. In Section \ref{sc moduli spaces of fixed points} we study these involution in the anti-holomorphic case and their relation with pseudo-real $G$-Higgs bundles. 

In Section \ref{sc HB on EC} we recall the case of elliptic curves. We study holomorphic involutions on elliptic curves in Section \ref{sc elliptic curves} and anti-holomorphic involutions in Section \ref{sc real EC}. We review the description of the moduli space of $G$-Higgs bundles over elliptic curves in Section \ref{sc higgs bundles on elliptic curves}.

The new results of this article are contained in Section \ref{sc description of 
involutions}. We first describe the involutions $I(\alpha_{\epsilon}, 
\sigma_{\epsilon}, \pm)$ on the moduli space of Higgs line bundles in Section 
\ref{sc rank 1} and using that, we obtain an explicit description of the involutions 
for the general case in Section \ref{sc general case}. This explicitness, allow us to 
describe their fixed point loci in Section \ref{sc description of fixed loci}, and in 
Section \ref{sc moduli of pseudo-real} we apply this description to the moduli space 
of pseudo-real $G$-Higgs bundles.

\section{Higgs bundles and involutions}\label{sec1}

\subsection{Involutions and conjugations of complex Lie groups}
\label{sc involutions}

Take $G$ to be a complex reductive affine algebraic group. A conjugation 
$\sigma_-\,:\,G\,\longrightarrow\, G$ is an anti-holomorphic automorphism of $G$ of 
order two. The set of all inner automorphisms of $G$ will be denoted by $\Conj(G)$. 
A real form of a conjugation $\sigma_-$ is the fixed point subgroup 
$G^{\sigma_-}\,\subset\, G$. There is a $\sigma_K\,\in\, \Conj(G)$ such that the 
real form $K= G^{\sigma_K}$ is a maximal compact subgroup of $G$ whose 
complexification is $G$.

Let $\Aut(G)$ be the group of all holomorphic automorphisms of $G$; denote 
by $\Int(G)$ the subgroup of $\Aut(G)$ consisting of inner automorphisms of $G$. This
is a normal subgroup of $\Aut(G)$ so one can consider the quotient group
$\Aut(G)/\Int(G)$ which we denote by $\Out(G)$. This all fits in the short
exact sequence
\begin{equation}\label{exacta out}
\xymatrix{
1\ar[r] &\Int(G)\ar[r] & \Aut(G)\ar[r] & \Out(G)\ar[r]&1.}
\end{equation}
Let $\Aut_{2}(G)$ the set of elements of order two in $\Aut(G)$. Two elements $f,\, f'$
of $\Aut_{2}(G)$ will be said to be  equivalent if there is some $\alpha\,\in\, \Int(G)$ such
that 
\[
f'\,=\,\alpha f \alpha^{-1}\,. 
\]
Similarly, we define an equivalence relation $\sim$ on $\Conj(G)$; in other words,
two anti-holo\-mor\-phic automorphisms of $G$ of order two are said to be equivalent if they
differ by conjugation by an inner automorphism.

\begin{theorem}[Cartan]\label{CT} There is a bijection:
\[
C\, :\, \Conj(G)/\!\!\sim\, \stackrel{1:1}{\longrightarrow}\,\Aut_2(G)/\!\!\sim\, ,
\ \ \sigma_{-}\,\longmapsto\, \sigma_{+}\,:=\,\sigma_{-}\sigma_K\, ,
\]
where $\sigma_K$ is the above conjugation.
\end{theorem}

\subsection{Involutions on $\Mm(G)$ and branes}
\label{sc brane involutions}

In this section $X$ will be a compact connected Riemann surface.
Let $\sigma_{+}\,:\, G \,\longrightarrow \,G$ be a holomorphic involution of the
complex Lie group $G$, and let $\sigma_{-} \,:\, G \,\longrightarrow \, G$ be an
anti-holomorphic involution.

Given a principal $G$-bundle $E$ on $X$, let $\sigma_{+}(E)$ 
(respectively, $\sigma_{-}(E)$) denote the principal $G$-bundle on $X$ obtained by 
extending the structure group of $E$ using $\sigma_{+}$ (respectively, 
$\sigma_{-}$). Note that the total spaces of both $\sigma_{+}(E)$ and 
$\sigma_{-}(E)$ have a natural holomorphic structure. In fact, $\sigma_{+}(E)$ is a 
holomorphic principal $G$-bundle, while $\sigma_{-}(E)$ is a holomorphic principal 
$\overline{G}$-bundle, where $\overline{G}$ is
the Lie group $G$ equipped with the almost complex structure $-J_G$ with $J_G$
being the almost complex structure on $G$.

For any holomorphic involution $\alpha_{+} \,:\, X \,\longrightarrow\, X$, define 
\[
I(\alpha_{+},\sigma_{+},+)\, :\, \Mm(G) \,\longrightarrow\,
\Mm(G)\, , \ \
(E,\,\varphi)\, \longmapsto\,
( \alpha_{+}^*\sigma_{+}(E),\, + \alpha_{+}^*\sigma_{+}(\varphi)) \, ,
\]
\[
I(\alpha_{+}, \sigma_{+},-)\, :\, \Mm(G)
\,\longrightarrow\, \Mm(G)\, , \ \
(E,\,\varphi)\, \longmapsto\,(
\alpha_{+}^*\sigma_{+}(E), \,- \alpha_{+}^*\sigma_{+}(\varphi))\, .
\]
For an anti-holomorphic involution $\alpha_{-} \,:\, X \,\longrightarrow\, X$, define
\[
I(\alpha_{-},\sigma_{-},+)\, :\, 
\Mm(G) \,\longrightarrow\,
\Mm(G) \, , \ \ (E,\,\varphi) \, \longmapsto\,
( \alpha_{-}^*\sigma_{-}(E), \,+\alpha_{-}^*\sigma_{-}(\varphi))\, ,
\]
\[
I(\alpha_{-},\sigma_{-},-)\, :\, 
\Mm(G) \,\longrightarrow\,
\Mm(G) \, , \ \ (E,\,\varphi) \, \longmapsto\,
( \alpha_{-}^*\sigma_{-}(E), \,-\alpha_{-}^*\sigma_{-}(\varphi))\, .
\]
Sometimes we write $I(\alpha_{\epsilon},\sigma_{\epsilon},\pm)$ when we want to deal 
with these involutions simultaneously. Note that $\epsilon$ indicates whether it is
the holomorphic ($\epsilon \,=\, +$) case or the anti-holomorphic ($\epsilon\,=\, 
-$) case.

See \cite{BGP, garcia-prada&ramanan} for a proof of the following proposition.

\begin{proposition}
If $\sigma_{\epsilon}\,\sim\, \sigma'_{\epsilon}$, then $I(\alpha_{\epsilon},\sigma_{\epsilon},\pm)\,=\,I(\alpha_{\epsilon},\sigma'_{\epsilon},\pm)$. 
\end{proposition}

Let $\Mm(G)_{\smooth}\, \subset\, \Mm(G)$ be the smooth locus. This manifold
$\Mm(G)_{\smooth}$ is hyper-K\"ahler with three complex structures $\Gamma_1,\Gamma_3,\Gamma_3$ and three associated K\"ahler forms
$\omega_1,\omega_2,\omega_3$. Also recall that $\Mm(G)_{\smooth}$ has three (holomorphic) symplectic
structures: $\Omega_1\,=\,\omega_2+\sqrt{-1}\omega_3,\, \Omega_2\,=\,\omega_3+\sqrt{-1}\omega_1$ and 
$\Omega_3\,=\,\omega_1+\sqrt{-1}\omega_2$. Now, we describe the geometric structure of the
fixed points of the involutions in $\Mm(G)_{\smooth}$ defined above (see \cite{BGP, garcia-prada&ramanan} for instance):
\begin{itemize}
\item the fixed point locus of $I(\alpha_{+},\sigma_{+},+)$ is an hyper-K\"ahler submanifold, 
\item the fixed point locus of $I(\alpha_{+},\sigma_{+},-)$ is a $\Omega_{1}$-Lagrangian submanifold,
\item the fixed point locus of $I(\alpha_{-},\sigma_{-},-)$ is a $\Omega_{2}$-Lagrangian submanifold,
\item the fixed point locus of $I(\alpha_{-},\sigma_{-},+)$ is a $\Omega_{3}$-Lagrangian submanifold.
\end{itemize}

\begin{remark}
In the context of Mirror Symmetry:
\begin{itemize}
\item the fixed point locus of $I(\alpha_{+},\sigma_{+},+)$ is a $(B,B,B)$-brane,
\item the fixed point locus of $I(\alpha_{+},\sigma_{+},-)$ is a $(B,A,A)$-brane,
\item the fixed point locus of $I(\alpha_{-},\sigma_{-},-)$ is a $(A,B,A)$-brane,
\item the fixed point locus of $I(\alpha_{-},\sigma_{-},+)$ is a $(A,A,B)$-brane.
\end{itemize}
\end{remark}

The above involutions  induce involutions of the moduli space of
representations defined in \eqref{rg}:
\[
J(\alpha_{+},\sigma_{+},+)\, :\, 
\Rr(G) \,\longrightarrow\, \Rr(G)\, , \ \ \rho \, \longmapsto\, \sigma_{+} \circ \rho \circ (\alpha_{+})_* \, ,
\]
\[
J(\alpha_{-},\sigma_{-},+)\, :\, 
\Rr(G) \,\longrightarrow\, \Rr(G)\, , \ \ \rho \, \longmapsto\, \sigma_{-} \circ \rho \circ (\alpha_{-})_* \, ,
\]
\[
J(\alpha_{+},\sigma_{+},-)\, :\, \Rr(G) \,\longrightarrow\, \Rr(G)
\, , \ \ \rho \, \longmapsto\, \sigma_{-} \circ \rho \circ (\alpha_{+})_*\, ,
\]
\[
J(\alpha_{-},\sigma_{-},-)\, :\, \Rr(G) \,\longrightarrow\, \Rr(G)
\, , \ \ \rho \, \longmapsto\, \sigma_{+} \circ \rho \circ (\alpha_{-})_*\, .
\]

Let $Z \,= \,Z(G)_0$ be the connected component of the center $Z(G)\,\subset\, G$ containing
the identity element. Define the homomorphism
$$\mu \,: \,Z \times G \,\longrightarrow\, G \, , \ \ (y\, ,z)\, \longmapsto\, yz\, . $$
For any principal $G$-bundle $E$ and any principal $Z$-bundle $F$, we have the principal
$G$-bundle
\[
F \otimes E \,:=\, \mu_*(F \times_X E)\, ;
\]
note that the fiber product $F \times_X E$ is a principal $(Z\times G)$-bundle and hence
$\mu_*(F \times_X E)$ is a principal $G$-bundle.
It is straight-forward to check that $F \otimes E$ is semistable, stable or polystable if
and only if $E$ is semistable, stable or polystable respectively. If $F'$ is a
principal $Z$-bundle, then using the multiplication operation
$\mu' \,: \,Z \times Z \,\longrightarrow\, Z$ we get a
principal $Z$-bundle
\[
F \otimes F' \,:=\, \mu'_*(F \times_X F')\, .
\]
This operation makes the moduli space $\Mm(Z)$ of topologically trivial principal
$Z$--bundles on $X$ a complex Lie group.

The Lie algebra of $Z$ will be denoted by $\zZ$. The adjoint bundle
$F(\zZ)$ for the principal $Z$-bundle $F$ is the trivial vector bundle 
$\Oo_X \otimes_{\mathbb C} \zZ$ over $X$ with fiber $\zZ$. The
Lie algebra of $G$ will be denoted by $\gG$. For the adjoint action
of $G$ on $\gG$, each point of $\zZ$ is fixed. Hence
$$
H^0(X,\, F(\zZ)) \,=\,\zZ\, \subset\, H^0(X,\, E(\gG))\, .
$$
The group $\Mm(Z)$ has the following holomorphic action on $\Mm(G)$:
\begin{equation} \label{eq action on Mm}
\Mm(Z) \times \Mm(G)\, \longrightarrow\, \Mm(G)\, , \ \ 
( (F,\,\phi),\, (E,\, \varphi)) \,\longmapsto\,  (F,\,\phi) \otimes (E,\, \varphi) = (F \otimes E,
\,\phi + \varphi)\, .
\end{equation}

Analogously, given a homomorphism $$\chi \,:\, \pi_1(X) \,\longrightarrow \,Z\, ,$$ one gets
for any homomorphism
$\rho \,:\, \pi_1(X) \,\longrightarrow  \,G$ a homomorphism $\chi \cdot \rho
\,:\, \pi_1(X) \,\longrightarrow \, G$
given by the composition $$\pi_1(X)\,\stackrel{\chi\times \rho}{\longrightarrow}\,
Z\times G \,\stackrel{\mu}{\longrightarrow}\, G\, .$$
More precisely, $\Rr(Z)$ (defined as in \eqref{rg}) is a complex Lie group that acts 
holomorphically
on $\Rr(G)$ as follows 
\begin{equation}\label{eq action on Rr}
\Rr(Z) \times \Rr(G)\, \longrightarrow\, \Rr(G)\, , \ \ (\chi,\, \rho)\, \longmapsto\,
\chi \cdot \rho\, .
\end{equation}
Note that if we set $G\, =\, Z$ in \eqref{rg}, then the GIT quotient becomes an
ordinary quotient.

Since the anti-holomorphic involution $\sigma_-$ of $G$ preserves $Z
\,\subset\, G$, we can combine \eqref{eq action on Mm} and \eqref{eq action on Rr} to
obtain more involutions. For every $\Ff \,=\, (F,\, \phi)\, \in\, \Mm(Z)$, set
\[
I(\alpha_{\epsilon},\sigma_{\epsilon},\pm, \Ff)\, :\,  \Mm(G) \, \longrightarrow\,
\Mm(G)\, , \ \ (E,\, \varphi) \, \longmapsto\, \Ff \otimes I(\alpha_{\epsilon},\sigma_{\epsilon},\pm)
(E,\varphi)\, .
\]
It is an involution whenever $\Ff \,=\, (F,\, \phi)$ satisfies
\begin{equation} \label{eq involution condition imath}
\Ff^{-1} \,=\, I(\alpha_{\epsilon},\sigma_{\epsilon},\pm)(F,\phi)\, ,
\end{equation}
where $\Ff^{-1}$ denotes $(F^{-1}, \,-\phi)$. The inverse of $F$ is denoted by
$F^{-1}$; note that $F^{-1}$ coincides with the principal $Z$--bundle obtained by extending
the structure group of $F$ using the automorphism $z\, \longmapsto\, z^{-1}$ of $Z$.

Analogously, for every $\chi \in \Rr(Z)$, define
\[
J(\alpha_{\epsilon},\sigma_{\epsilon},\pm, \chi)\, :\, \Rr(G) \, \longrightarrow\,
\Rr(G)\, , \ \ \rho \, \longmapsto\, \chi \cdot J(\alpha_{\epsilon},\sigma_{\epsilon},\pm)(\rho)\, .
\]
The condition for $J_{\alpha_{\epsilon}}^{\sigma_{\epsilon}, \chi}$ to be an involution is that
\begin{equation} \label{eq involution condition jmath}
\chi^{-1} \,=\, J(\alpha_{\epsilon},\sigma_{\epsilon},\pm)(\chi)\, .
\end{equation}

\begin{theorem} \label{tm involution branes}
For all $\Ff \,\in\, \Mm(Z)$ and $\chi \,\in\, \Rr(Z)$ satisfying
\eqref{eq involution condition imath} and \eqref{eq involution condition jmath},
the above maps $I(\alpha_{\epsilon},\sigma_{\epsilon},\pm, \Ff)$ and $J(\alpha_{\epsilon},\sigma_{\epsilon},\pm, \chi)$ are involutions.

The fixed points of $I(\alpha_{+},\sigma_{+},+, \Ff)$ restricted to $\Mm(G)_{\smooth}$ are $(B,B,B)$-branes, while the fixed points of $I(\alpha_{+},\sigma_{+},-, \Ff)$ are $(B,A,A)$-branes. On the other hand, the fixed points of $I(\alpha_{-},\sigma_{-},-, \Ff)$ are $(A,B,A)$-branes when we restrict to $\Mm(G)_{\smooth}$ and the fixed points of $I(\alpha_{-},\sigma_{-},+, \Ff)$ are $(A,A,B)$-branes.

If $\Ff \,=\, (F,\phi) \,\in \,\Mm(Z)$ is the $Z$-Higgs bundle associated to the representation
$\chi \in \Rr(Z)$, then the diagram
\begin{equation} \label{eq HK commutes with i_ell}
\xymatrix{
\Mm(G) \ar[rr]^{\rm homeo.} \ar[d]_{I(\alpha_{\epsilon},\sigma_{\epsilon},\pm, \Ff)} & & \Rr(G) \ar[d]^{J(\alpha_{\epsilon},\sigma_{\epsilon},\pm, \chi)}
\\
\Mm(G) \ar[rr]^{\rm homeo.}  & & \Rr(G)
}
\end{equation}
is commutative.
\end{theorem}

\begin{proof}
The statements for $I(\alpha_{\epsilon},\sigma_{\epsilon},\pm)$ follow from the 
works \cite{BS}, \cite{BGP} and \cite{garcia-prada&ramanan}. One can prove that the 
diffeomorphism in Theorem \ref{ths} takes the map in \eqref{eq action on Mm} to the 
map in \eqref{eq action on Rr}. Then $I(\alpha_{\epsilon},\sigma_{\epsilon},\pm, 
\Ff)$ commutes or anticommutes with $\Gamma_1$ and $\Gamma_2$ whenever 
$I(\alpha_{\epsilon},\sigma_{\epsilon},\pm)$ does so. By the above references, the 
involutions $I(\alpha_{\epsilon},\sigma_{\epsilon},\pm)$ preserve the hyper-K\"ahler 
metric on $\Mm(G)_{\smooth}$. Note that the translation \eqref{eq action on Mm} 
preserves the hyper-K\"ahler metric as well, and so does 
$I(\alpha_{\epsilon},\sigma_{\epsilon},\pm, \Ff)$. As a consequence, the statements 
proved for $I(\alpha_{\epsilon},\sigma_{\epsilon},\pm)$ are also valid for 
$I(\alpha_{\epsilon},\sigma_{\epsilon},\pm, \Ff)$.
\end{proof}

\begin{remark} \label{rm HK commutes with alpha_0}
When $\sigma_+ \,=\, \id_G$, the involution $I(\alpha_{+},\id_G,+)$ is just the 
pull-back by $\alpha_{+}$, while $J(\alpha_{+},\id_G,+)$ coincides with the 
composition of homomorphisms with $(\alpha_+)_*$. Denote by $\id_X^{-1}$ and $\id_G^{-1}$ the inversion 
given by the group structures on $X$ and $G$, respectively. For any representation 
of the fundamental group, note that $\rho \circ (\id_X^{-1})_* \,= \,\rho^{-1} = 
\id^{-1}_{G} \circ \rho$. As a consequence of the commutativity of \eqref{eq HK 
commutes with i_ell}, the pull-back by $\id_X^{-1}$ commutes with the extension of 
structure group associated to $\id^{-1}_{G}$.
\end{remark}

\subsection{Pseudo-real Higgs bundles}
\label{sc moduli spaces of fixed points}

Let $\alpha_- \,:\, X \,\longrightarrow \, X$ be an anti-holomorphic involution of
$X$ and $\sigma_{-}\,:\,G\, \longrightarrow\, G$ an anti-holomorphic involution of
$G$. Let $z\,\in\, Z_{2}^{\sigma_{-}}\, \subset\, Z^{\sigma_{-}}$ be an element
of order 2 of the fixed point locus of $Z^{\sigma_{-}}\, \subset\, Z$ under $\sigma_{-}$. A pseudo-real $(G,\alpha_{-},\sigma_{-},\pm,z)$-Higgs bundle is a $G$-Higgs bundle $(E,\,\varphi)$ over $X$ equipped with a lift
\[
\xymatrix{
E \ar[r]^{\widetilde{\alpha}_{-}} \ar[d] & E \ar[d]
\\
X \ar[r]^{\alpha_{-}} & X \, ,
}
\]
satisfying the following conditions:
\begin{itemize}
\item $\widetilde{\alpha}_{-}$ is anti-holomorphic; 
\item $\widetilde{\alpha}_{-}(eg)\,=\,
\widetilde{\alpha}_{-}(e)\sigma_{-}(g)$, for all $e\,\in\, E$, $g\,\in\, G$;
\item $\widetilde{\alpha}_{-}^{2}(e)\,=\,e z$, for $e\,\in\, E$;
\item $\widetilde{\alpha}_{-}(\varphi)\,=\,\pm \varphi$.
\end{itemize}
The last condition needs an explanation. The canonical line bundle $K$ of $X$
has an anti-holomorphic involution induced by $\alpha_{-}$. The bundle 
$E(\mathfrak{g})$ has also an anti-holomorphic involution given by $d\sigma_{-}$
and $\alpha_{-}$. Using this two involutions together, we can define $$\widetilde{\alpha}_{-}\,:\,E(\mathfrak{g}) \otimes K \,\longrightarrow\, E(\mathfrak{g})\otimes K\, .$$

The $(G,\alpha_{-},\sigma_{-},\pm,z)$-Higgs bundles are also called pseudo-real 
$G$-bundles; these are called real $G$-Higgs bundles if further $z = \id_G$. These are studied in 
\cite{BGP} and \cite{BGH}; the reader is referred to \cite{BGP}, \cite{BGH}
for the introduction of the moduli space 
$$\Mm(G,\alpha_{-},\sigma_{-},\pm,z)$$ of isomorphism classes of polystable 
$(G,\alpha_{-},\sigma_{-},\pm,z)$-Higgs bundles.

Note that the holomorphic isomorphism $\widetilde{\alpha}_{-} \,:\, E
\,\longrightarrow\, E$ and $\alpha_{-}$ produce a holomorphic isomorphism of $G$-Higgs bundles
\begin{equation} \label{eq theta}
\theta \,:\, (E,\, \varphi) \,\stackrel{\cong}{\longrightarrow}\,
( \alpha_{-}^*\sigma_{-}(E), \,\pm \alpha_{-}^*\sigma_{-}(\varphi)) \, ,
\end{equation}
such that 
\begin{equation} \label{eq theta^2 = z}
\theta \circ \alpha_{-}^* \sigma_{-}(\theta) \, = \, z \cdot \id_E \, .
\end{equation}
We will refer to the triples $(E,\, \varphi,\, \theta)$ too as $(G, \alpha_{-},
\sigma_{-}, \pm)$-Higgs bundles. An isomorphism of $(G,\alpha_{-}, \sigma_{-},
\pm, z)$-Higgs bundles $(E_1, \varphi_1, \widetilde{\alpha}_1)$ and
$(E_2,\, \varphi_2,\, \widetilde{\alpha}_2)$ is an isomorphism of Higgs bundles
$$f \,:\, (E_1,\, \varphi_1)
\,\stackrel{\cong}{\longrightarrow}\, (E_2,\, \varphi_2)$$ such that $f^{-1} \circ \widetilde{\alpha}_2 \,=\, \widetilde{\alpha}_1 \circ f$, or, equivalently,
\begin{equation} \label{eq isomorphic pseudo-real}
\theta_2 \, = \, f \circ \theta_1 \circ \alpha_{-}^* \sigma_{-}(f)^{-1} \, .
\end{equation}

We consider the forgetful map $\Mm(G,\alpha_{-},\sigma_{-},\pm,z)
\,\longrightarrow\, \Mm(G)$. Let 
$\widetilde{\Mm}(G,\alpha_{-},\sigma_{-},\pm,z)$ be the image of this forgetful map.

\begin{theorem}[\cite{BGP}]\label{th puntos fijos}
The set $$ \bigcup_{z\in Z_{2}^{\sigma_{-}}}\widetilde{\Mm}
(G,\alpha_{-},\sigma_{-},\pm,z)$$ is contained in the fixed
points of $I(\alpha_{-},\sigma_{-},\pm)$ in $\Mm(G)$. 
\end{theorem}

\begin{remark}
In the case of holomorphic involutions on the curve, one could consider as well $(G,\alpha_+, \sigma_+, \pm)$-Higgs bundles and we should obtain a similar description of the fixed point locus as in Theorem \ref{th puntos fijos}. For the trivial involution $\sigma_+ = \id_X$, these objects would be equivalent to $(G_{\sigma_+}, \pm)$-Higgs bundles as described in \cite{garcia-prada&ramanan}. 
For the trivial involution $\sigma_+ = \id_G$, $(G,\alpha_+, \id_G, +)$-Higgs 
bundles have been studied in \cite{garcia-prada&wilkin}.

See \cite{FGN2} for a concrete description of the fixed point loci of the involutions $I(\id_X, \sigma_+, -)$ in the moduli space $\Mm(G)$ of $G$-Higgs bundles over elliptic curves.
\end{remark}

Fix a point $x\,\in\, X$ such that $\alpha_{-}(x)\,\neq\, x$. Let $$p\,:\,\widetilde{X}\,\longrightarrow\, X$$ be the universal cover of $X$ associated to $x$. The orbifold fundamental
group $\Gamma(X,x)$ of $(X,\alpha_{-})$ is the set
$$\{f:\widetilde{X}\rightarrow \widetilde{X}, p\circ f\,=\,q(f)\circ p\}\, ,$$ where $q(f)$ can be $\sigma_{-}$. So, we have a short exact sequence
$$
\xymatrix{
0\ar[r] &\pi_1(X,x)\ar[r]^{i} &\Gamma(X, x)\ar[r]^{q} &\mathbb{Z}/2\mathbb{Z}\ar[r]&0}\, .
$$

Let $\widehat{G}$ be $G\times(\mathbb{Z}/2\mathbb{Z})$ as a set, with the
group operation: 
$$(g_1,\,e_1)(g_2,\,e_2)\,=\,(g_1^{e_1}g_2 c^{e_{1} e_{2}},\, e_{1}+e_{2})\, .$$

A $z$-homomorphism of the orbifold fundamental group is a group homomorphism $$\rho\,:\,
\Gamma(X, x)\,\longrightarrow\, \widehat{G}$$ that fits in the following diagram:
\begin{equation}
\xymatrix{
0\ar[r] &\pi_1(X,x)\ar[r]^{i}\ar[d] &\Gamma(X, x)\ar[r]^{q}\ar[d]^{\rho} &\mathbb{Z}/2\mathbb{Z}\ar[r]\ar[d]^{=}&0\\
0\ar[r] &G\ar[r]^{i} &\widehat{G}\ar[r]^{q} &\mathbb{Z}/2\mathbb{Z}\ar[r]&0
}
\end{equation}

Two $c$-homomorphism $\rho,\rho'$ are called equivalent if there is an element $g\,\in\, G$ with $\rho'\,=\,g\rho g^{-1}$. 
Consider the moduli space of equivalence classes of reductive representations
$\Rr(G,\alpha_{-},\sigma_{-},\pm,)$ \cite{BGP}. Let $\Rr(G)_{\smooth}$ be the smooth locus corresponding to $\Mm(G)_{\smooth}$. Let $\widetilde{\Rr}(G,\alpha_{-},\sigma_{-},\pm,z)$ be the image of the forgetful map
$\Rr(G,\alpha_{-},\sigma_{-},\pm,z)\,\longrightarrow\,
\Rr(G)$.

\begin{theorem}\mbox{}
The set $$ \bigcup_{z\in Z_{2}^{\sigma_{-}}}\Rr(G,\alpha_{-},\sigma_{-},\pm,z)$$ is contained in the fixed points of $J(\alpha_{-}, \sigma_{-},\pm)$ in $\Rr(G)$.
\end{theorem}

\proof
It is a consequence of Theorem \ref{th puntos fijos} and \cite[Theorems 4.5 and 4.8]{BGP}.

\begin{remark}
In \cite{garcia-prada&wilkin}, a similar description is worked out for the holomorphic case, defining $(G, \alpha_+, 
\sigma_+, \pm, z)$--Higgs bundles and their moduli spaces $\Mm(G, \alpha_+, \sigma_+, \pm, z)$, as well as the 
associated moduli spaces of representations $\Rr(G, \alpha_+, \sigma_+, \pm, z)$.
\end{remark}

\section{Elliptic curves and Higgs bundles}
\label{sc HB on EC}

\subsection{Elliptic curves}\label{sc elliptic curves}

Let $X$ be a compact Riemann surface of genus $1$, and let $x_0$ be a distinguished 
point on it; the pair $(X,\,x_0)$ defines an {\it elliptic curve} (a connected
compact complex Lie group of dimension one). However, by abuse of 
notation, we usually denote the elliptic curve simply by $X$.
Every elliptic curve is isomorphic to some $X_\gamma \,=\, \CC/ \langle 1,\, \gamma \rangle_\ZZ$, where $\langle 1,\, \gamma \rangle_\ZZ$ is the lattice generated by $1$ and $\gamma\,\in\, \HH$, the subset of points of the complex plane $\CC$ with
positive imaginary part. One has $X_{\gamma_1} \cong X_{\gamma_2}$ if and only if
$\gamma_1, \gamma_2 \in \HH$ are related by a M\"obius transformation given by
an element of $\text{SL}(2, {\mathbb Z})$.

The Picard group $\Pic^{0}(X)$ will be denoted by $\widehat{X}$. Let 
\begin{equation} \label{eq definition of p_x0} p_{x_0} \,:\, X\, 
\stackrel{\cong}{\longrightarrow}\, \widehat{X} \end{equation} be the holomorphic 
isomorphism that sends any $x\,\in\, X$ to the holomorphic line bundle ${\mathcal 
O}_X(x-x_0)$. Note that the group structure on $\widehat{X}$ produces a group 
structure on $T^*\widehat{X}$. Recall that $T^*\widehat{X} \,\cong \,\widehat{X} 
\times H^0(X,\, \Oo_X)$; the group structure on the second factor is simply given 
by the linear structure. The subgroup of $X$ defined by the points of order two 
will be denoted by $X[2]$.

The Abel-Jacobi morphism 
\[
\aj_1 \,:\, X \,\longrightarrow\, \Pic^1(X)\, , \ \ x \, \longmapsto\, {\mathcal O}_X(x)
\]
is an isomorphism. For every $y \,\in\, X$, let
\[
t_y\, :\, X \,\longrightarrow\, X\, , \ \ x\, \longmapsto\,x+y
\]
be the translation automorphism.

To present a consistent notation with that of Section \ref{sc real EC}, we write 
$\alpha_{(+, 1)}$ for the identity $\id_X$ and $\alpha_{(+,-1)}$ for
the inversion map $x\, \longmapsto\, -x$ of $X$. When they are simultaneously
referred, we shall write $\alpha_{(+, a)}$. Composing $\alpha_{(+,a)}$ with the 
translations $t_y$ one gets another involution provided $y\,=\, - \alpha_{(+,a)}(y)$. 
For an involution $\alpha$ of $X$, let
$X^\alpha\, \subset \, X$ be the fixed point locus and $X/\alpha$ the quotient by
$\alpha$. Table \ref{tb holomorphic involutions} describes the holomorphic involutions on $X$.

For $\gamma\,\in\,\mathbb H$, consider the paths $\widetilde{\delta}_1,\, \widetilde{\delta}_2 \,:\, [0,\,1]
\,\longrightarrow\, \CC$, where
\begin{equation} \label{eq definition of delta_1}
\widetilde{\delta}_1 (t) \,=\, t
\end{equation}
and
\begin{equation} \label{eq definition of delta_2}
\widetilde{\delta}_2 (t) = \gamma t\, .
\end{equation}
Let $\delta_1, \,\delta_2$ be the images of
$\widetilde{\delta}_1,\, \widetilde{\delta}_2$ in $X_\gamma \,=\, \CC/ \langle 1,\, \gamma
\rangle_\ZZ$. The homotopy classes of $\delta_1$ and $\delta_2$ generate the
fundamental group $\pi_1(X_\gamma,\, 0)$. The involutions $\alpha_{(+,a)}$ induce
an action on the fundamental group; this action is described in in Table \ref{tb action of alpha on pi_1 hol}.

\begin{remark} \label{rm homogenicity_+}
Every semistable principal $G$-bundle $E\,\longrightarrow\, X$ of trivial topological
type is homogeneous, meaning
\[
t_y^* E \,\cong\, E
\]
for every $y \,\in\, X$ \cite[Theorem 4.1]{BG}. This implies that the pull-back by
$\alpha_{(+, a)}$ coincides with the pull-back by $t_y \circ \alpha_{(+, a)}$;
therefore, both involutions define the same involution on the moduli space of
$G$--bundles of trivial topological type, meaning
\[
I(t_y \circ \alpha_{(+,a)}, \sigma_+, \pm) = I(\alpha_{(+,a)}, \sigma_+, \pm). 
\]
Since the fundamental group of $X$ is abelian,
$\pi_1(X, x_0)$ is identified with $\pi_1(X, y)$ (in general they are identified
uniquely up to an inner automorphism); with this identification,
the action of $(t_y)_*$ on $\pi_1(X)$ is trivial. Therefore,
\[
J(t_y \circ \alpha_{(+,a)}, \sigma_+, \pm) = J(\alpha_{(+,a)}, \sigma_+, \pm). 
\]
\end{remark}

\subsection{Real elliptic curves}\label{sc real EC}

A {\it real elliptic curve} is a triple of the form $(X,\,\alpha_-,\,x_{0})$, 
where $(X,\,x_0)$ is an elliptic curve, and $\alpha_-$ is an anti-holomorphic 
involution of $X$. It should be emphasized that the point $x_0$ need not be fixed by
$\alpha_-$. Now, again, by 
abuse of notation, we will denote a real elliptic curve simply by 
$(X,\,\alpha_-)$. A morphism of real elliptic curves is a holomorphic group homomorphism $X 
\,\longrightarrow\, Y$ that intertwines the involutions.

A Klein surface is a pair consisting on compact Riemann surface and an anti-holomorphic involution. The topological classification of Klein surfaces was given by Klein \cite{K}. In the particular case of real elliptic curves, the topological type of a real elliptic curve is given 
by a pair $(n,b)$ where $n$ is the number of connected components of the fix point locus $\alpha_-$, and $b\,\in\, 
\mathbb{Z}/2\mathbb{Z}$ is the index of orientability defined by
\begin{equation*}
b=
 \begin{cases}
    0&  \text{if $X/\alpha_-$ is oriented }\\
    1&  \text{if $X/\alpha_-$ is not oriented}.
  \end{cases}
\end{equation*}
A theorem of Harnack says that $n \, \leq\, \text{genus}(X)+1$ \cite{Ha}.

There are three different topological types of real elliptic curves: a Klein bottle, a 
M\"obius strip and a closed annulus; see Table \ref{zzt}.

Consider the elliptic curve $X_\gamma \,=\, \CC/\langle 1, \gamma \rangle_\ZZ$ for
some $\gamma \,\in\, \HH$, and let $\alpha_-$ be an anti-holomorphic involution on it.
Fix a lift
$$\widetilde{\alpha}_-\,:\,\mathbb{C}\,\longrightarrow \,\mathbb{C}\, ,\ \
z\, \longmapsto\, a\overline{z}+b$$ of
$\alpha_-$, where $a,\, b$ are complex numbers. Table \ref{tb region of partial Omega}
gives the possible values of $\gamma$ up to M\"obius transformations.

Continuing with the notation of Section \ref{sc elliptic curves}, an
anti-holomorphic involution on $X\,=\, X_\gamma$ that fixes the identity element and
lifts to the automorphism $z\,\longmapsto\, a \overline{z}$ of $\CC$, where
$a \,\in\, \CC^*\,=\, \CC\setminus\{0\}$, will be denoted by $\alpha_{(-,a)}$. The anti-holomorphic involution $t_y \circ \alpha_{(-,a)}$ lifts to $z\,\longmapsto\,
a\overline{z} + b$, where $y \,\in\, X$ is the projection of $b$ to
$X_\gamma$. Consider the two involutions $\alpha_{(\epsilon_1, a_1)}$ and
$\alpha_{(\epsilon_2, a_2)}$; note that
\[
\alpha_{(\epsilon_1, a_1)} \circ \alpha_{(\epsilon_2, a_2)} \,=\, \alpha_{(\epsilon_1\epsilon_2, a_1a_2)}
\]
if $\epsilon_1 \,=\, +$, and 
\[
\alpha_{(\epsilon_1, a_1)} \circ \alpha_{(\epsilon_2, a_2)} \,=\, \alpha_{(\epsilon_1\epsilon_2, a_1\overline{a}_2)}
\]
when $\epsilon_1 \,=\, -$.

All possible anti-holomorphic involutions are contained 
in Table \ref{tb all the involutions} (see \cite[Section 9]{alling&greenleaf}).

Alling and Greenleef gave the following classification of real tori.

\begin{proposition}[{\cite[Section 9]{alling&greenleaf}}]
In the cases in Table \ref{tb all the involutions}, for any possible $y \,\in\, X$,
the isomorphisms of real tori are as follows:
\begin{itemize}
\item In the region $A$, for any $y \,\in\, X^{\alpha_{(-, - 1)}}$ such
that $y \neq x_0$,
\[
\left( X_\gamma, \,t_y \circ \alpha_{(-,1)} \right) \,\cong\,
\left( X_\gamma, \,t_{\frac{1}{2}} \circ \alpha_{(-,1)} \right)
\]
and, for any $y \,\in\, X^{\alpha_{(-, 1)}}$ with $y \neq x_0$,
\[
\left( X_\gamma, \,t_y \circ \alpha_{(-,-1)} \right) \,\cong\, \left( X_\gamma,\, t_{\frac{\gamma}{2}} \circ \alpha_{(-,1)} \right).
\]

\item In the region $B$, 
\[
\left( X_\gamma, \,\alpha_{(-,1)} \right) \,\cong\, \left( X_\gamma,\, \alpha_{(-,-1)} \right),
\]
\[
\left( X_\gamma, \,\alpha_{(-,\i)} \right) \,\cong\, \left( X_\gamma,\, \alpha_{(-,-\i)} \right)
\]
and for every $y \in X^{\alpha_{(-, \mp 1)}}$ such that $y \neq x_0$,
\[
\left( X_\gamma, \,t_y \circ \alpha_{(-,1)} \right) \,\cong\, \left( X_\gamma,\, t_{\frac{1}{2}} \circ \alpha_{(-,1)} \right).
\]

\item In the regions $C$, $D$ and $E$, for all $y \in X^{\alpha_{(-, \mp a)}}$,
\[
\left( X_\gamma, \,t_y \circ \alpha_{(-,a)} \right) \,\cong \,\left( X_\gamma,\, \alpha_{(-,a)} \right).
\]

\item In the region $D$,
\[
\left( X_\gamma, \,\alpha_{(-,\pm 1)} \right) \,\cong\, \left( X_\gamma, \,\alpha_{(-,\mp \gamma)} \right) \cong \left( X_\gamma, \alpha_{(-,\pm \gamma^2)} \right).
\]
\end{itemize}
\end{proposition}

Recall that $\delta_1, \delta_2$, defined in \eqref{eq definition of delta_1} and 
\eqref{eq definition of delta_2} respectively, are generators of the fundamental group 
$\pi_1(X)$. Table \ref{tb action of alpha on pi_1 antihol} describes the action of $\alpha_{(\epsilon,a)}$
on $\pi_1(X)$.
 
\begin{remark} \label{rm homogenicity_-}
As in Remark \ref{rm homogenicity_+}, the homogeneity of topologically trivial semistable principal $G$-bundles on
elliptic curves implies that the pullbacks by $\alpha_{(-, a)}$ and
$t_y \circ \alpha_{(-, a)}$ of such a bundle are isomorphic. We also have the identifications 
\[
I(t_y \circ \alpha_{(-,a)}, \,\sigma_-, \pm) \,=\, I(\alpha_{(-,a)}, \,\sigma_-, \pm) 
\]
and
\[
J(t_y \circ \alpha_{(-,a)}, \,\sigma_-, \pm) \,=\, J(\alpha_{(-,a)},\, \sigma_-, \pm) 
\]
in the anti-holomorphic case. The notation used here is guided by these identities.
\end{remark}

\subsection{Higgs bundles over elliptic curves}\label{sc higgs bundles on elliptic curves}

As above, let $G$ be a connected complex reductive affine algebraic group. Let $T\, \subset\, G$ be
a Cartan 
subgroup, and let $\Lambda_{T} \,:=\, \Hom(\CC^*,\, T)$ be the corresponding cocharacter 
lattice. We denote by $Z(G)$ the center of $G$, and by $Z$ the connected component of it containing 
the identity element. Fix, once and for all, a basis
$$\{ \lambda_1, \cdots, \lambda_\ell, \lambda_{\ell + 1}, \cdots, \lambda_s \}$$
of $\Lambda_T$ such that $\{ \lambda_1, \cdots,\lambda_\ell \} \oplus
\{ \lambda_{\ell + 1}, \cdots, \lambda_s \}$ is an orthogonal decomposition of it,
with $\{\lambda_1, \cdots, \lambda_\ell \}$ being an orthogonal basis of
$\Lambda_Z \,:= \,\Hom(\CC^*,\, Z)$. We consider the natural isomorphism
\[
\eta\, :\, \CC^*\otimes_{\ZZ} \Lambda_{T}\, \stackrel{\cong}{\longrightarrow}\,
T\, ,\ \ \sum z_i \otimes_{\ZZ} \lambda_i\,\longmapsto\, \Pi \lambda_i(z_i)\, .
\]

Let $\mu \,:\, T \times T \,\longrightarrow\, T$ be the multiplication map of the group
$T$. For any two principal $T$-bundles $E$ and $E'$ of trivial topological type,
consider the principal $(T \times T)$-bundle $E \times_X E'$, and define
\[
E \otimes E' \,:=\, \mu_*(E \times_X E')\, ,
\]
which is again a principal $T$-bundle of trivial topological type. Note that
$E(\tT) \,\cong\, \tT \otimes \Oo_X$, so Higgs fields on
a principal $T$-bundle are elements of $\tT$. One can express
the extension of structure groups associated to $\eta$ as follows
\[
\eta_*\, :\, T^*\widehat{X} \otimes_\ZZ \Lambda_T
\, \stackrel{\cong}{\longrightarrow}\,\Mm(T)\, , \ \
\sum (L_i, \psi_i) \otimes \lambda_i\,\longmapsto\,\left( \bigotimes(\lambda_i)_*L_i \thinspace,
\thinspace \sum (d \lambda_i)_* \psi_i \right)\, .
\]
Consider also the extension of structure group associated to the injection
$T \,\hookrightarrow\, G$, and compose it with $\eta_*$, to obtain the following map:
\[
\dot{\xi} \,:\, T^*\widehat{X} \otimes_\ZZ \Lambda_T \,\longrightarrow\, \Mm(G)\, .
\]
Analogously, one can define 
\[
\dot{\zeta}\, :\,\Hom(\pi_1(X),\CC^*) \otimes_\ZZ \Lambda_T\,\longrightarrow\,
\Rr(G)\, ,\ \ \sum_i \rho_i \otimes \lambda_i \,\longmapsto\,\prod_i
(\lambda_i \circ \rho_i)\, .
\] 

The action of the Weyl group $W$ associated to $(T,\, G)$ on $\Lambda_T$ can be extended
to $A \otimes_\ZZ \Lambda_T$, where $A$ is any abelian group, as follows:
\[
\omega \cdot \left( \sum a_i \otimes \lambda_i \right) \,=\,
\sum a_i \otimes \omega(\lambda_i)\, , \ \ \omega\, \in\, W\, ,\ \ a_i\,\in\, A\, .
\]
If $A$ is the multiplicative group $\CC^*\,=\, \CC \setminus\{0\}$, this action commutes with the natural
action of $W$ on $T$,
\begin{equation} \label{eq W commutes with Gamma}
\xymatrix{
\CC^* \otimes_\ZZ \Lambda_T \ar[rr]^{\quad \eta} \ar[d]_{\omega \cdot} & & T \ar[d]^{\omega \cdot}
\\
\CC^* \otimes_\ZZ \Lambda_T \ar[rr]^{\quad \eta} & & T.
}
\end{equation}
The commutativity of \eqref{eq W commutes with Gamma} implies that $\dot{\xi}$ and
$\dot{\zeta}$ factor through the quotient by the action of $W$.

\begin{theorem}[{\cite{FGN}}] \label{tm HB on EC}
The morphism $\dot{\xi}$ and $\dot{\zeta}$ induce isomorphisms
\[
\xi \,:\, (T^* \widehat{X} \otimes_\ZZ \Lambda_T)/W \,\stackrel{\cong}{\longrightarrow}\,
\Mm(G)
\]
and 
\[
\zeta \,:\, (\Hom(\pi_1(X),\CC^*) \otimes_\ZZ \Lambda_T)/W\,
\stackrel{\cong}{\longrightarrow}\, \Rr(G).
\]

Furthermore, the diffeomorphism between $T^*\widehat{X}$ and $\Hom(\pi_1(X),\CC^*)$ induced by
the Hodge theory induces the top row map of the commuting diagram
\begin{equation} \label{eq HK commutes with xi and zeta}
\xymatrix{
T^*\widehat{X} \otimes_\ZZ \Lambda_T \ar[rr]^{\rm diffeo. \quad \quad} \ar[d]_{\dot{\xi}} & & \Hom(\pi_1(X),\CC^*) \otimes_\ZZ \Lambda_T \ar[d]^{\dot{\zeta}}
\\
\Mm(G) \ar[rr]^{\rm diffeo.} & & \Rr(G)
}
\end{equation}
\end{theorem}

\section{Description of the involutions of the moduli space}
\label{sc description of involutions}

\subsection{The rank $1$ case}
\label{sc rank 1}

The identity map of the multiplicative group $\CC^*$
is associated to the anti-holomorphic involution of the compact real form $\U(1)$
\begin{equation} \label{eq definition of sigma_U}
\sigma_{\U(1)}\, :\, \CC^*\,\longrightarrow\,\CC^*\, , \ \ 
z\,\longmapsto\, \overline{z}^{-1}\, .
\end{equation}

For the group $G \,=\, \CC^*$, set 
\[
i(\alpha_{(+,a)},\, \pm) \,:=\, I(\alpha_{(+,a)},\,
\id_{\CC^*}, \,\pm)
\]
and 
\[
i(\alpha_{(-,a)},\, \pm) \,:=\, I(\alpha_{(-,a)},\, \sigma_{\U(1)},\,
\pm). 
\]
Assume that $\Ff \,=\, (F,\,\phi)$ satisfies the involution condition in
\eqref{eq involution condition imath}. Then we can define the following involutions on $T^*\widehat{X}$
\[
i(\alpha_{(+,a)}, \pm, \Ff)\, :\, T^*\widehat{X}\,\longrightarrow\,
T^*\widehat{X}\, , \ \ (L,\,\psi)\,\longmapsto\, (F \otimes \alpha_{(+,a)}^*L,\,
\phi \pm \alpha_{(+,a)}^*\psi)\, ,
\]
and
\[
i(\alpha_{(-,a)}, \pm, \Ff)\, :\,T^*\widehat{X}\,\longrightarrow\,T^*\widehat{X}
\, , \ \ (L,\,\psi)\,\longmapsto\, (F \otimes \alpha_{(-,a)}^*
\overline{L}^*,\, \phi \mp \alpha_{(-,a)}^*\overline{\psi})\, .
\]
Accordingly, set $j(\alpha_{(+,a)}, \pm) \,=\, J(\alpha_{(+,a)}, \id_{\CC^*}, \pm)$ and $j(\alpha_{(-,a)}, \pm) \,=\, J(\alpha_{(-,a)}, \sigma_{\U(1)}, \pm)$. Simi\-larly, if
$\chi$ satisfies \eqref{eq involution condition jmath}, consider the involutions on $\Hom(\pi_1(X),\,\CC^*)$
\[
j(\alpha_{(\epsilon,a)}, +, \chi)\,:\, \Hom(\pi_1(X),\,\CC^*)\,\longrightarrow\,\Hom(\pi_1(X),\,\CC^*)
\, , \ \ \rho\,\longmapsto\, \chi \left( \cdot \rho
\circ (\alpha_{(\epsilon,a)})_* \right)
\]
and
\[
j(\alpha_{(\epsilon,a)}, -, \chi)\,:\, \Hom(\pi_1(X),\,\CC^*)\,\longrightarrow\,\Hom(\pi_1(X),\,\CC^*)
\, , \ \ \rho\,\longmapsto\, \chi \cdot
( \overline{\rho}^{-1}\circ \left(\alpha_{(\epsilon,a)})_*\right)\, .
\]

Since $T^*\widehat{X} \,\cong\, \widehat{X} \times H^0(X,\, \Oo_X)$, the
isomorphism $p_{x_0}$ in \eqref{eq definition of p_x0} produces an isomorphism
\[
p_{x_0} \,:\, X \times H^0(X, \,\Oo_X) \,\stackrel{\cong}{\longrightarrow}\,
T^*\widehat{X}
\]
(it is a slight abuse of notation to use $p_{x_0}$ for this map).
Given an anti-holomorphic involution $\alpha \,:\, X \,\longrightarrow\, X$, let
\[
\conj\, :\, H^0(X, \,\Oo_X)\,\longrightarrow\,H^0(X, \,\Oo_X)\, , \ \
X\,\longmapsto\, \alpha^*\overline{X}
\]
be the induced conjugate-linear homomorphism.

\begin{lemma} \label{lm description of the i}
Take $\Ff \,=\, (F,\,\phi)$ satisfying \eqref{eq involution condition imath}, and let $y \,\in \,X$ be such that $F = p_{x_0}(y)$. Then,
\[
i(\alpha_{(+,a)}, \pm, \Ff) \,=\, p \circ (t_y \circ \alpha_{(+,a)},\, \pm \id ) \circ p^{-1}\, ,
\]
and
\[
i(\alpha_{(-,a)}, \pm, \Ff)\,=\, p \circ (t_y \circ \alpha_{(-,-a)},\, \mp \conj ) \circ p^{-1}\, .
\]
\end{lemma}

\begin{proof}
The first statement is straight-forward. The lemma follows from the fact that for all $x\,\in\, X_{\gamma}$,
\begin{align*}
\alpha_{(-,a)}^{*}\sigma_{\U(1)} \left ( \Oo_X(x-x_0) \right) \,=\, & \alpha_{(-,a)}^{*}\overline{\Oo(x-x_0)}^* 
\\
=\, & \Oo_X(-\alpha_{(-,a)}(x) - x_0) 
\\
= \,& \Oo_X(\alpha_{(-,-a)}(x) - x_0)\, .
\end{align*}
\end{proof}

Recall the generators $\delta_1, \delta_2$ of $\pi_1(X)$ defined in \eqref{eq definition of delta_1} and \eqref{eq definition of delta_2}. One has the isomorphism
\[
q\,:\, \Hom(\pi_1(X),\, \CC^*)\,\longrightarrow\, \CC^* \times \CC^*\, ,\ \
\rho\, \longmapsto\, \left( \rho(\delta_1),
\rho(\delta_2) \right)\, .
\]
Note that the involution $j(\alpha_{(-,a)}, \pm, \chi)$ is not defined for every isomorphism
class of elliptic curves; rather it is defined only for certain values specified in Table \ref{tb region of partial Omega}. For each region $R$ of Table \ref{tb region of partial Omega}, or for
the entire upper-half plane $\HH$ when $\epsilon\,=\, 1$, let
\[
f_{(\epsilon,a,R)}^\pm \, :=\, q \circ j(\alpha_{(\epsilon,a)}, \pm) \circ q^{-1}
\]
be the automorphism of $\CC^* \times \CC^*$.

\begin{remark} \label{rm description of the j}
Given $\chi \,\in \,\Rr(Z)$, set $b_i \,=\, \chi(\delta_i)$. Note that $\chi$ satisfies \eqref{eq involution condition jmath} if and only if 
\[
(b_1^{-1}, b_2^{-1}) \,=\, f^{\pm}_{(\epsilon,a,R)}(b_1, b_2)\, .
\]
If this holds, then
\[
f^{\pm, (b_1, b_2)}_{(\epsilon,a,R)}\, :\, \CC^* \times \CC^*\,\longrightarrow\,\CC^* \times \CC^*
\, , \ \ (z_1, \,z_2)\, \longmapsto\,(b_1,b_2)
\cdot f^{\pm}_{(\epsilon,a,R)}(z_1, \, z_2)
\]
are involutions, and
\[
j(\alpha_{(\epsilon,a)}, \pm, \chi) \,=\, q^{-1} \circ f^{\pm, (b_1, b_2)}_{(\epsilon,a,R)} \circ q\, .
\]
\end{remark}

For any involution $\alpha_{(\epsilon,a)}$ of the curve $X$,
using Tables \ref{tb action of alpha on pi_1 hol} and \ref{tb action of alpha on pi_1 antihol}, we
can describe $f_{(\epsilon,a,R)}^\pm$, and therefore we can
describe $f^{\pm, (b_1, b_2)}_{(\epsilon,a,R)}$; see Table \ref{tb values of f^pm}.

\subsection{The general case}
\label{sc general case}

Let $\sigma_{+}$ and $\sigma_{-}$ respectively be the Cartan (holomorphic)
involution and the associated real form (anti-holomorphic involution) of 
$G$. Recall from Theorem \ref{CT} that they are related through the composition by the compact real form involution $\sigma_K$
commuting with $\sigma_{-}$, 
\begin{equation} \label{eq sigma_+ = sigma_- sigma_0}
\sigma_{+} \,=\, \sigma_{-} \sigma_K\, .
\end{equation}
Given a Cartan subgroup $T$ of $G$ preserved by $\sigma_{+}$, $\sigma_{-}$ and $\sigma_K$, we denote by the same symbols the involutions in the cocharacter lattice
$\Lambda_T \,=\, \Hom(\CC^*,\, T)$,
\begin{equation} \label{eq definition of sigma_+ on Lambda}
\sigma_{+}\,:\, \Lambda_T\,\longrightarrow\,\Lambda_T
\, , \ \ \lambda\, \longmapsto\,\sigma_{+} \circ \lambda
\end{equation}
and
\begin{equation} \label{eq definition of sigma_- on Lambda}
\sigma_{-}\,:\, \Lambda_T\,\longrightarrow\,\Lambda_T
\, , \ \ \lambda\, \longmapsto\,\sigma_{-} \circ \lambda \circ \sigma_{\U(1)}\, .
\end{equation}
Define an action of $\sigma_K$ on $\Lambda_T$ by $\sigma_K(\lambda)\,= \,
\sigma_K \circ \lambda \circ \sigma_{\U(1)}$; note that the equality
\eqref{eq sigma_+ = sigma_- sigma_0}, considered as an equality of involutions of $\Lambda_T$,
is recovered.

Since the Cartan subgroup $T$ is preserved by $\sigma_+$, $\sigma_-$ and 
$\sigma_K$, these involutions induce involutions of the normalizer $N_G(T)$, and 
therefore produce involutions of the Weyl group $W \,=\, N_{G}(T)/Z_{G}(T)$;
these involutions of the Weyl group are denoted by the same symbol.

\begin{remark} \label{rm sigma_+ = sigma_- on Lambda}
Let $T_0 \,:=\, T^{\sigma_K}$ be the compact torus fixed pointwise by
$\sigma_K$. Since $T_0^\CC \,=\, T$, it follows that $\Hom(\CC^*,\,T)
\,\cong\, \Hom(\U(1),\, T_0)$. The action of $\sigma_K$ is trivial on
$\Lambda_T \,\cong\, \Hom(\U(1),\, T_0)$. Also, we recall that $\sigma_K$ acts
trivially on $W$, because it is a compact real form. As a consequence,
the actions of $\sigma_{+}$ and $\sigma_{-}$ on $\Lambda_T$ coincide. The
same statement holds for the actions of $\sigma_{+}$ and $\sigma_{-}$ on $W$.
\end{remark}

\begin{remark} \label{rm description of sigma_+}
The Vogan diagram is constructed with the root data of a maximally compact Cartan 
subalgebra $\tT_\RR$ of a simple real Lie algebra $\gG_\RR$ and it consists on a 
triple $V \,=\, (D, \theta, S)$ where $D$ is the Dynkin diagram of $\gG \,=\, (\gG_\RR)^\CC$, 
$\theta$ is an automorphism of the diagram given by the Cartan involution $\sigma_{+}$ 
and $S$ is a subset (possibly empty) of the vertices of $D$ fixed by $\theta$.

The Vogan diagram encodes the action of $\sigma_{\epsilon}$ on
$\Lambda_T$ when $T$ is the complexification of a maximal compact subgroup of
$G_\RR$, the real subgroup fixed by $\sigma_{-}$.

Recall form Remark \ref{rm sigma_+ = sigma_- on Lambda} that it suffices to 
describe the action of $\sigma_{+}$ on $\Lambda_T$. The automorphism $\theta$ is an 
involution on the set of simple roots, so we obtain a description on the entire set of 
roots. Translate this involution to the coroot lattice $\Delta$ by setting 
$\theta(\alpha^{\vee}) \,= \,(\theta(\alpha))^{\vee}$, one gets an involution on 
$\tT \cong \CC \otimes_\ZZ \Delta$ which is precisely the Cartan involution 
$\sigma_{+}$. The restriction of $\sigma_{+}$ to $\Lambda \,\subset\, \tT$ coincides with 
\eqref{eq definition of sigma_+ on Lambda}.
\end{remark}

As done in \eqref{eq definition of sigma_+ on Lambda} and
\eqref{eq definition of sigma_- on Lambda}, define the holomorphic involution
\[
\dot{\sigma}_+\,:\, \CC^*\otimes_{\ZZ} \Lambda_{T}\,\longrightarrow\,
\CC^*\otimes_{\ZZ} \Lambda_{T}\, , \ \ \sum z_i \otimes_{\ZZ} \lambda_i\,\longmapsto\,
\sum z_i \otimes_{\ZZ} \sigma_{+}(\lambda_i)
\]
and the anti-holomorphic involution
\[
\dot{\sigma}_-\,:\, \CC^*\otimes_{\ZZ} \Lambda_{T}\,\longrightarrow\,
\CC^*\otimes_{\ZZ} \Lambda_{T}\, , \ \ \sum z_i \otimes_{\ZZ} \lambda_i\,\longmapsto\,
\sum \sigma_{\U(1)}(z_i) \otimes_{\ZZ} \sigma_{-}(\lambda_i)\, .
\]
We use again $\dot{\sigma}_\epsilon$ to refer simultaneously $\dot{\sigma}_+$ and $\dot{\sigma}_-$. Now the following diagram is commutative
\begin{equation} \label{eq Gamma commutes with sigma}
\xymatrix{
\CC^*\otimes_{\ZZ} \Lambda_{T} \ar[d]_{\dot{\sigma}_{\epsilon}}\ar[rr]^{\quad \eta} & & T \ar[d]^{\sigma_{\epsilon}}
\\
\CC^*\otimes_{\mathbb{Z}} \Lambda_{T}\ar[rr]^{\quad \eta} & & T.}
\end{equation}

The action of $W$ on $\Lambda_T$ is $\sigma_{\epsilon}$-equivariant, meaning for any $\omega \,\in \, W$ and any $\lambda \in \Lambda_T$,
\begin{equation} \label{eq the Weyl group sigma equivariant}
\sigma_{\epsilon}(\omega \cdot \lambda) \,= \,\sigma_{\epsilon}(\omega) \cdot \sigma_{\epsilon}(\lambda)\, .
\end{equation}
Therefore, we have
\begin{equation} \label{eq the Weyl group dot sigma equivariant}
\dot{\sigma}_\epsilon \circ \omega \,=\, \sigma_{\epsilon}(\omega) \circ \dot{\sigma}_\epsilon\, .
\end{equation}

Now define the involutions 
\[
\dot{I}(\alpha_{(\epsilon, a)},\sigma_{\epsilon},\pm)\,:\, T^*\widehat{X} \otimes_\ZZ \Lambda_T\,\longrightarrow\,
T^*\widehat{X}\otimes_\ZZ \Lambda_T\, , 
\] 
\[  
\sum (L_i,\psi_i)\otimes \lambda_i\,
\longmapsto\, \sum i(\alpha_{(\epsilon, a)},\pm)(L_i,\psi_i) \otimes \sigma_{\epsilon}(\lambda_i)\, ,
\]
and
\[
\dot{J}(\alpha_{(\epsilon, a)},\sigma_{\epsilon},\pm)\,:\,\Hom(\pi_1(X),\CC^*) \otimes_\ZZ \Lambda_T \longrightarrow\,\Hom(\pi_1(X),\CC^*) \otimes_\ZZ \Lambda_T\, , \ \
\]
\[
\sum \rho_i \otimes \lambda_i\,\longmapsto\,
\sum j(\alpha_{(\epsilon, a)},\pm)(\rho_i) \otimes \sigma_{\epsilon}(\lambda_i)\, .
\]

Fix any $\Ff \,= \,(F,\, \phi)\,\in\, \Mm(Z)$; let
$$\chi \,:\, \pi_1(X) \,\longrightarrow\, Z$$
be the corresponding representation of the fundamental group. We recall that
\[
\Mm(Z) \,\cong\, T^*\widehat{X} \otimes_\ZZ \Lambda_Z\, ,\ \
\Rr(Z) \,\cong\, \Hom(\pi_1(X), \CC^*) \otimes_\ZZ \Lambda_Z\, .
\]
Let $[\Ff] \,=\, \dot{\xi}^{-1}(\Ff)$ and $[\chi] \,=\, \dot{\zeta}^{-1}(\chi)$
be the corresponding elements of the groups $T^*\widehat{X} \otimes_\ZZ \Lambda_T$
and $\Hom(\pi_1(X), \CC^*) \otimes_\ZZ \Lambda_T$ respectively. Following the
definitions of $I(\alpha_{(\epsilon, a)},\sigma_{\epsilon},\pm, \Ff)$ and $J(\alpha_{(\epsilon, a)},\sigma_{\epsilon},\pm, \Ff)$, define 
\[
\dot{I}(\alpha_{(\epsilon, a)},\sigma_{\epsilon},\pm, \Ff)\,:\, T^*\widehat{X} \otimes_\ZZ \Lambda_T
 \,\longrightarrow\,T^*\widehat{X} \otimes_\ZZ \Lambda_T\, , 
\] 
\[  
\sum (L_i,\psi_i) \otimes \lambda_i\,\longmapsto\,
[\Ff] + \dot{I}(\alpha_{(\epsilon, a)},\sigma_{\epsilon},\pm) \left ( \sum z_i
\otimes_\ZZ \lambda_i \right )
\]
and
\[
\dot{J}(\alpha_{(\epsilon, a)},\sigma_{\epsilon},\pm, \chi)\,:\,\Hom(\pi_1(X),\CC^*) \otimes_\ZZ \Lambda_T \longrightarrow\,\Hom(\pi_1(X),\CC^*) \otimes_\ZZ \Lambda_T\, , \ \
\]
\[
\sum \rho_i \otimes \lambda_i\,\longmapsto\,[\chi] + \dot{J}(\alpha_{(\epsilon, a)},\sigma_{\epsilon},\pm) \left ( \sum \rho_i \otimes \lambda_i \right )\, .
\]

\begin{lemma}
The diagrams
\begin{equation}\label{eq i commutes with xi}
\xymatrix{
T^*\widehat{X} \otimes_\ZZ \Lambda_T \ar[rr]^{\dot{\xi}} \ar[d]_{\dot{I}(\alpha_{(\epsilon, a)},\sigma_{\epsilon},\pm, \Ff)} & & \Mm(G) \ar[d]^{I(\alpha_{(\epsilon, a)},\sigma_{\epsilon},\pm, \Ff)}
\\
T^*\widehat{X} \otimes_\ZZ \Lambda_T \ar[rr]^{\dot{\xi}} & & \Mm(G)
}
\end{equation}
and
\begin{equation} \label{eq j commutes with zeta}
\xymatrix{
\Hom(\pi_1(X),\CC^*) \otimes_\ZZ \Lambda_T \ar[rr]^{\quad \quad \dot{\zeta}} \ar[d]_{\dot{J}(\alpha_{(\epsilon, a)},\sigma_{\epsilon},\pm, \Ff)} & & \Rr(G) \ar[d]^{J(\alpha_{(\epsilon, a)},\sigma_{\epsilon},\pm, \Ff)}
\\
\Hom(\pi_1(X),\CC^*) \otimes_\ZZ \Lambda_T \ar[rr]^{\quad \quad \dot{\zeta}} & & \Rr(G)
}
\end{equation}
commute.
\end{lemma}

\begin{proof}
In view of \eqref{eq the Weyl group dot sigma equivariant} it follows that
$\dot{I}(\alpha_{(\epsilon, a)},\sigma_{\epsilon},\pm, \Ff)$ induces an involution on the quotient $T^*\widehat{X} \otimes_\ZZ \Lambda_T / W$. From
\eqref{eq Gamma commutes with sigma} and the construction of $\xi$ it is
clear that this induced involution coincides with $I(\alpha_{(\epsilon, a)},\sigma_{\epsilon},\pm, \Ff)$. 

The proof of the commutativity of \eqref{eq j commutes with zeta} is analogous.
\end{proof}

In Section \ref{sc higgs bundles on elliptic curves} we chose a basis
of $\Lambda_T$ with an orthogonal decomposition
$$\{ \lambda_1, \cdots, \lambda_\ell \} \oplus \{ \lambda_{\ell + 1}, \cdots,
\lambda_s \}\, ,$$
where $\{ \lambda_1, \cdots, \lambda_\ell \}$ is an orthogonal basis of
$\Lambda_Z \,=\, \Hom(\CC^*,\, Z)$. Take a $Z$-Higgs bundle $\Ff \,=\, (F, \,\phi)$, and
set $\Ff_i^{\epsilon} \,=\, (F^{\epsilon}_i, \phi^{\epsilon}_i)$ to be the Higgs line bundles such that
\begin{equation} \label{eq condition on Jj_i}
[\Ff] \,=\, \sum [\Ff_i^\epsilon] \otimes \sigma_{\epsilon}(\lambda_i)\, .
\end{equation}
Let us extend this by setting $\Ff^{\epsilon}_i \,:=\, (\Oo_X,\, 0)$ for
$\ell+1 \,\leq\, i \,\leq\, s$. Analogously, we take a decomposition of the
representation $\chi \,:\, \pi_1(X) \,\longrightarrow\, Z$ 
\begin{equation} \label{eq condition on chi_i}
[\chi] \,= \,\sum [\chi_i^\epsilon] \otimes \sigma_{\epsilon}(\lambda_i) 
\end{equation}
where $\chi^{\epsilon}_i$ are homomorphisms from $\pi_1(X)$ to $\CC^*$.
As before, set $\chi^{\epsilon}_i \,=\, \id$ for $\ell+1\,\leq\, i\,\leq\, s$.

\begin{proposition} \label{pr description of the I and J in therms of i and j}
For $\Ff$ and $\chi$ as above, construct $\Ff_i^{\epsilon}$ and $\chi_i^{\epsilon}$
as in \eqref{eq condition on Jj_i} and \eqref{eq condition on chi_i} respectively. Then
\[
\dot{I}(\alpha_{(\epsilon, a)},\sigma_{\epsilon},\pm, \Ff) \left ( \sum (L_i,\psi_i) \otimes \lambda_i \right )\,=\,
\sum i(\alpha_{(\epsilon, a)},\pm, \Ff^{\epsilon}_i)(L_i,\psi_i)
\otimes \sigma_{\epsilon}(\lambda_i)
\]
and
\[
\dot{J}(\alpha_{(\epsilon, a)},\sigma_{\epsilon},\pm, \Ff) \left ( \sum \rho_i \otimes \lambda_i \right )\,=\,
\sum j(\alpha_{(\epsilon, a)},\pm, \chi^{\epsilon}_i)(\rho_i)
\otimes \sigma_{\epsilon}(\lambda_i)\, .
\]
\end{proposition}

\begin{proof}
Once we have the descriptions of $\dot{I}(\alpha_{(\epsilon, a)},\sigma_{\epsilon},\pm)$ and
$\dot{J}(\alpha_{(\epsilon, a)},\sigma_{\epsilon},\pm)$, it is immediate to derive descriptions of
$\dot{I}(\alpha_{(\epsilon, a)},\sigma_{\epsilon},\pm, \Ff)$ and $\dot{J}(\alpha_{(\epsilon, a)},\sigma_{\epsilon},\pm, \chi)$. We have 
\begin{align*}
\dot{I}(\alpha_{(\epsilon, a)},\sigma_{\epsilon},\pm, \Ff) \left ( \sum (L_i,\psi_i) \otimes \lambda_i \right )
\,=\,& \sum [\Ff_i^\epsilon] \otimes \sigma_{\epsilon}(\lambda_i) + \sum i(\alpha_{(\epsilon, a)},\pm)(L_i,\psi_i) \otimes \sigma_{\epsilon}(\lambda_i)
\\
\,=\,& \sum \left ( [\Ff_i^\pm] + i(\alpha_{(\epsilon, a)},\pm)(L_i,\psi_i) \right ) \otimes \sigma_{\epsilon}(\lambda_i)
\\
\,=\,& \sum i(\alpha_{(\epsilon, a)},\pm, \Ff^\epsilon_i)(L_i,\psi_i) \otimes \sigma_{\epsilon}(\lambda_i)\, .
\end{align*}
The case of $\dot{J}(\alpha_{(\epsilon, a)},\sigma_{\epsilon},\pm, \chi)$ is analogous.
\end{proof}

\subsection{Description of the fixed point locus}
\label{sc description of fixed loci}

Consider the cocharacter lattice $\Lambda_T$ of a complex reductive Lie group $G$. 
Denote by $\sigma_{-}$ the anti-holomorphic involution of a certain 
real form, and denote by $\sigma_{+}$ the Cartan involution associated to it. Recall
from \eqref{eq the Weyl group sigma equivariant} that the Weyl group $W$ acts 
$\sigma_{\epsilon}$-equivariantly on $\Lambda_T$.

Let $A$ be a complex abelian group; construct the tensor product 
\[
\dot{B} \, :=\, A \otimes_\ZZ \Lambda_T\, .
\]
Consider also the action of $W$ on $\dot{B}$ induced by the action of $W$ on $\Lambda$
and take its quotient
\[
B \,:=\, \dot{B}/W\, . 
\]
Fix a basis $\{ \lambda_1, \cdots, \lambda_s \}$ of $\Lambda$ and suppose that
$A$ is equipped with a set of analytic group homomorphisms of order two $\{t_1, \cdots, t_s \}$
\[
t_i \,:\, A \,\longrightarrow\, A
\]
such that we can combine them with $\sigma_{\epsilon}$ to obtain
the analytic involution
\[
\dot{\tau}\,:\,\dot{B} \,=\, A \otimes_\ZZ \Lambda \,\longrightarrow\,
\dot{B}\, ,\ \ \sum a_i\otimes \lambda_i \,\longmapsto\,\sum t_i(a_i) \otimes
\sigma_{\epsilon}(\lambda_i)\, .
\]
Since the action of the Weyl group is $\sigma_{\epsilon}$-equivariant, it follows that
$\dot{\tau}$ induces an involution on the quotient
\[
\tau\,:\,B\,=\, (A \otimes_\ZZ \Lambda)/W\,\longrightarrow\,B\, ,\ \ 
\left[ \sum a_i \otimes \lambda_i \right]_W \,\longmapsto\,\left[
\sum t_i(a_i) \otimes \sigma_{\epsilon}(\lambda_i) \right]_W\, .
\]
Evidently, the diagram
\[
\xymatrix{
A \otimes_\ZZ \Lambda_T \ar[r]^{p} \ar[d]_{\dot{\tau}} & A \otimes_\ZZ \Lambda_T/W \ar[d]^{\tau}
\\
A \otimes_\ZZ \Lambda_T \ar[r]^{p} & A \otimes_\ZZ \Lambda_T/W
}
\]
commutes, where $p$ is the projection induced by the quotient map
for the action of $W$. By construction, we have a similar commuting diagram for
every element $\omega \in W$,
\begin{equation} \label{eq tau commutes with the projection}
\xymatrix{
A \otimes_\ZZ \Lambda_T \ar[r]^{p} \ar[d]_{\omega\dot{\tau}} & A \otimes_\ZZ \Lambda_T/W \ar[d]^{\tau}
\\
A \otimes_\ZZ \Lambda_T \ar[r]^{p} & A \otimes_\ZZ \Lambda_T/W
}
\end{equation}

The aim of this section is to describe the fixed point set
$(A \otimes_\ZZ \Lambda_T/W)^\tau$. To do so, we will make use of the commutativity of
\eqref{eq tau commutes with the projection}. Define, for any $\omega \,\in\, W$, the
subgroup of $A \otimes_\ZZ \Lambda_T/W$
\[
(A \otimes_\ZZ \Lambda_T/W)^{\tau}_\omega \,:=\,
p((A \otimes_\ZZ \Lambda_T)^{\omega\dot{\tau}})\, .
\]
One has that
\[
(A \otimes_\ZZ \Lambda_T/W)^\tau \,=\, \bigcup_{\omega \in W}
(A \otimes_\ZZ \Lambda_T/W)^\tau_\omega\, .
\]

\begin{remark}
Note that each $\dot{B}^{\omega \dot{\tau}}$ is a closed subset, because it is
a fixed point locus of an involution. The projection $p$ is continuous, and therefore
$B^\tau_\omega \,= \, p(\dot{B}^{\omega \dot{\tau}})$ is closed as well.
\end{remark}

\begin{lemma} \label{lm coboundary condition}
The equality 
\[
(A \otimes_\ZZ \Lambda_T/W)^{\tau}_{\omega_1} \,=\,
(A \otimes_\ZZ \Lambda_T/W)^{\tau}_{\omega_2}
\]
holds if and only if there exists an element $\omega' \,\in \,W$ such that 
\begin{equation}\label{eq coboundary condition}
\omega_2 \,=\, \omega' \omega_1 \sigma_{\epsilon}(\omega')^{-1}\, .
\end{equation}
\end{lemma}

\begin{proof}
We have that $x' \,\in\, (A \otimes_\ZZ \Lambda_T)^{\omega\dot{\tau}}$ if and only if
$x'\,= \,\omega \dot{\tau}(x')$. Now $x\,= \,\omega' x'$ lies in
$\omega' \cdot (A \otimes_\ZZ \Lambda_T)^{\omega\dot{\tau}}$ if and only if
\[
(\omega')^{-1} x \,=\, \omega \dot{\tau}((\omega')^{-1} x)\, .
\]
In that case,
\[
x \,=\, \omega' \omega \sigma_{\epsilon}(\omega')^{-1} \dot{\tau}(x)\, .
\]
We see that for $\omega, \omega' \,\in\, W$,
\[
\omega' \cdot (A \otimes_\ZZ \Lambda_T)^{\omega\dot{\tau}}
\,=\, (A \otimes_\ZZ \Lambda_T)^{\omega' \omega \sigma_{\epsilon}(\omega')^{-1}\dot{\tau}}\, ,
\]
which implies the lemma. 
\end{proof}

Following \eqref{eq coboundary condition}, define the $\sigma_{\epsilon}$-adjoint action
\[
\ad_{\sigma_{\epsilon}}\, :\, W \times W\,\longrightarrow\, W\, , \ \
(\omega',\, \omega)\,\longmapsto\,\omega' \omega \sigma_{\epsilon}(\omega')^{-1}\, .
\]
{}From Lemma \ref{lm coboundary condition} it follows that the components are
parametrized by 
\[
W /_{\sigma_{\epsilon}} W \,=\, W/\ad_{\sigma_{\epsilon}}(W)\, .
\]
Therefore, we write
\[
(A \otimes_\ZZ \Lambda_T/W)^\tau \,=\, \bigcup_{\overline{\omega}
\in W /_{\sigma_{\epsilon}} W} (A \otimes_\ZZ \Lambda_T/W)^\tau_{\omega}\, ,
\]
where $\omega$ is a representative of the $\sigma_{\epsilon}$-conjugacy class
$\overline{\omega} \,\in\, W /_{\sigma_{\epsilon}} W$.

\begin{proposition} \label{pr descrition of B^tau_omega}
Denote by $T^{\omega \sigma_{\epsilon}}$ the subtorus fixed by $\omega \sigma_{\epsilon}$.
Then
\begin{equation} \label{eq B^tau_omega =  dotB^tau_omega quotiented by Z_W^tau}
(A \otimes_\ZZ \Lambda_T/W)^{\tau}_{\omega} \,\cong\, 
((A \otimes_\ZZ \Lambda_T)^{\omega\dot{\tau}})/N_W(T^{\omega \sigma_{\epsilon}})\, .
\end{equation}
Furthermore,
\begin{equation} \label{eq dimension of B^tau_omega}
\dim(B^\tau_\omega) \,=\, (\dim(A \otimes_\ZZ \Lambda_T))/2 \cdot \ord(\omega
\sigma_{\epsilon}(\omega))\, .
\end{equation}
\end{proposition}

\begin{proof}
We claim that
\begin{equation} \label{eq N_W T}
N_W(T^{\omega \sigma_{\epsilon}}) \,=\, \{ \omega' \,\in\, W \,\mid\, \omega
\,= \,\omega' \omega \sigma_{\epsilon}(\omega')^{-1} \}.
\end{equation}
To prove this claim, first recall that 
\[
N_W(T^{\omega \sigma_{\epsilon}}) \,=\, \{ \omega' \in W\,\mid\,
\omega'(T^{\omega\sigma_{\epsilon}}) = T^{\omega \sigma_{\epsilon}} \}\, .
\]
Now note that
\[
\omega'(T^{\omega \sigma_{\epsilon}}) \,=\, T^{\omega'\omega\circ \sigma_{\epsilon}\circ (\omega')^{-1}}
\,=\, T^{\omega' \omega \sigma_{\epsilon}(\omega')^{-1} \circ \sigma_{\epsilon}}\, .
\]
We have $T^{\omega' \omega \sigma_{\epsilon}(\omega')^{-1} \circ \sigma_{\epsilon}}
\,=\, T^{\omega\sigma_{\epsilon}}$ if and only if
$\omega' \omega \sigma_{\epsilon}(\omega')^{-1} \,= \,\omega$. This proved the claim.

The first statement in the proposition follows from the combination of
Lemma \ref{lm coboundary condition} and the above claim.

To prove the second statement, since $\dot{\tau}$ is an analytic homomorphism, so is
$\omega\dot{\tau}$. We assume that they are all nontrivial (so we exclude the case of
$\dot{\tau}$ being an element of $W$). Since $p$ is the quotient by a finite subgroup,
one has
\[
\dim((A \otimes_\ZZ \Lambda_T/W)^\tau_\omega) \,= \,
\dim((A \otimes_\ZZ \Lambda_T)^{\omega \dot{\tau}})\, ,
\]
and therefore
\[
\dim((A \otimes_\ZZ \Lambda_T/W)^\tau_\omega) \,=\, (\dim(A \otimes_\ZZ \Lambda_T))/
\ord(\omega \dot{\tau})\, .
\]
Note that
\begin{equation} \label{eq omega tau^2}
\omega \dot{\tau} \omega \dot{\tau} \,=\, \omega \sigma_{\epsilon}(\omega)\dot{\tau}^2
\,=\, \omega \sigma_{\epsilon}(\omega)\, , 
\end{equation}
so $\ord(\omega \dot{\tau}) \,=\, 2 \cdot \ord(\omega \sigma_{\epsilon}(\omega))$.
\end{proof}

We now proceed to describe the fixed locus $\dot{B}^{\omega \sigma_{\epsilon}}$. Using the
basis $\{ \lambda_1, \cdots, \lambda_s \}$ of $\Lambda$ one gets an isomorphism
between $\dot{B}$ and
$\overset{s-\text{times}}{\overbrace{A \times \cdots \times A}}$. 
Denote by $M \in \GL(s,\ZZ)$ the matrix of $\omega \sigma_{\epsilon}$ in this base. We denote by $\overline{t} = (t_1, \cdots, t_s)$ the involution on $A^s$ given by the $t_i$. We observe that $\omega \dot{\tau}$ corresponds with $M \circ \overline{t}$. It is then clear that
\begin{equation} \label{eq description of dotB^omega sigma}
(A \otimes_\ZZ \Lambda_T)^{\omega \sigma_{\epsilon}}\,
\cong\, \left(\overset{s-\text{times}}{\overbrace{A \times \cdots \times A}} \right )^{M \circ \overline{t}}.
\end{equation}

The following is a consequence of Propositions \ref{pr description of the I and J in therms of i and j} and \ref{pr descrition of B^tau_omega}, Lemma \ref{lm description of the i} and Remarks \ref{rm homogenicity_+}, \ref{rm homogenicity_-} and \ref{rm description of the j}.

\begin{corollary} \label{co description of the fixed locus of I and J}
Let $t_y \circ \alpha_{(\epsilon, a)}$ be a holomorphic (respectively, anti-holomorphic) involution on $X$ and $\sigma_\epsilon$ a holomorphic (respectively, anti-holomorphic) involution of the complex reductive Lie group $G$. The fixed locus for the involution $I(t_y \circ \alpha_{(\epsilon, a)},\sigma_{\epsilon},\pm, \Ff)$ of $\Mm(G)$ is the union 
\[
\Mm(G)^{I(t_y \circ \alpha_{(\epsilon, a)},\sigma_{\epsilon},\pm, \Ff)} \,=\, \bigcup_{\overline{\omega} \in
W/_{\sigma_{\epsilon}} W} \left( T^*\widehat{X} \otimes_\ZZ \Lambda_T \right)^{\omega
\left ( \dot{I}(\alpha_{(\epsilon, a)},\sigma_{\epsilon},\pm, \Ff) \right)}/N_W(T^{\omega \sigma_{\epsilon}})\, , 
\]
where $\omega$ is a representative of $\overline{\omega}\,\in\, W/_{\sigma_{\epsilon}} W$.
Furthermore, if $M_\omega \,\in\, \GL(s, \ZZ)$ is the automorphism
$\omega \sigma_{\epsilon}(\omega)$ expressed in a certain basis of $\Lambda_T$, and if
$y_1, \cdots, y_s \,\in\, T^*\widehat{X}$ are the coordinates of $\Ff \,\in\,
T^*\widehat{X} \otimes_\ZZ \Lambda_T$ in this basis, then
\[
\left( T^*\widehat{X} \otimes_\ZZ \Lambda_T \right)^{\omega
\left (\dot{I}(\alpha_{(\epsilon, a)},\sigma_{\epsilon},\pm, \Ff) \right)}
\,\cong\, \left(\overset{s-{\rm times}}{\overbrace{T^*\widehat{X} \times {\cdots} \times
T^*\widehat{X}}} \right)^{M_\omega \circ \overline{t}}\, ,
\]
where
\[
\overline{t} \,=\, \left((\nu_{y_1} \circ \alpha_{(+,a)}, \pm \id), \cdots ,
(\nu_{y_s} \circ \alpha_{(+,a)} , \pm \id) \right)
\]
if $\epsilon\,=\, +$, and 
\[
\overline{t} \,=\, \left( (\nu_{y_1} \circ \alpha_{(-,-a)} , \mp \conj), \cdots,
(\nu_{y_s} \circ \alpha_{(-,-a)} , \mp \conj) \right)
\]
if $\epsilon \,=\, -$.

Analogously, the fixed locus for the involution $J(t_y \circ \alpha_{(\epsilon, a)},\sigma_{\epsilon},\pm, \chi)$ of
$\Rr(G)$ is the union 
\[
\Rr(G)^{I(t_y \circ \alpha_{(\epsilon, a)},\sigma_{\epsilon},\pm, \chi)} \,=\, \bigcup_{\overline{\omega} \in W/_{\sigma_{\epsilon}} W}
 \left( \Hom(\pi_1(X), \,\CC^*) \otimes_\ZZ \Lambda_T \right)^{\omega \left
( \dot{J}(\alpha_{(\epsilon, a)},\sigma_{\epsilon},\pm, \chi) \right)}/N_W(T^{\omega \sigma_{+}})\, . 
\]
If the homomorphisms $\chi_i \,\in\, \Hom(\pi_1(X),\, \CC^*)$ are the coordinates of
$\chi \,\in\, \Hom(\pi_1(X), \,\CC^*) \otimes_\ZZ \Lambda_T$ in the basis, and if
$b_{i,1} \,:=\, \chi_i(\delta_1)$ and $b_{i,2}\,:=\, \chi_i(\delta_2)$ are the images
of the generators $\delta_1, \delta_2$ of $\pi_1(X)$, then
\[
\left( \Hom(\pi_1(X), \,\CC^*) \otimes_\ZZ \Lambda_T \right)^{\omega \left(
\dot{J}(\alpha_{(\epsilon, a)},\sigma_{\epsilon},\pm, \chi) \right)} \,\cong \,\left ((\CC^* \times \CC^*)
\times \stackrel{s}{\cdots} \times (\CC^* \times \CC^*) \right)^{M_\omega
\circ \overline{t}}\, ,
\]
where
\[
\overline{t} \,=\, \left( f^{\pm, (b_{1,1}, b_{1,2})}_{(\epsilon,a,R)}, \cdots , f^{\pm,
(b_{s,1}, b_{s,2})}_{(\epsilon,a,R)} \right )
\]
with $R$ being the region of $\Omega$ where $X$ lies. 
\end{corollary}

It is straight-forward to check that the subset of $W$ given by the elements of the
form $\omega \sigma_{\epsilon}(\omega)$ is closed under the usual adjoint action of $W$. We define the quotient
\[
\Upsilon_{\sigma_{\epsilon}} \,=\, \{ \omega \sigma_{\epsilon}(\omega) \ \mid \ \omega \in W \}/{\ad(W)}\, .
\]
Define the map
\[
\delta\, :\, W /_{\sigma_{\epsilon}} W\, \longrightarrow\, \Upsilon_{\sigma_{\epsilon}}\, ,\ \
[\omega]_{\ad_{\sigma_{\epsilon}}}\,\longmapsto\, [\omega\sigma_{\epsilon}(\omega)]_\ad\, .
\]
In view of \eqref{eq dimension of B^tau_omega}, the dimension of $B^\tau_\omega$ is determined by $\delta([\omega]_{\ad_{\sigma_{\epsilon}}})$.

Following \eqref{eq dimension of B^tau_omega}, we say that a
component $B^\tau_\omega$ is {\it maximal} is the order of $\omega \dot{\tau}$ is $2$. It is clear that the dimension of each maximal component is half the dimension of $B$. We see that a component is maximal when
\begin{equation} \label{eq cocycle condition}
\id_W \,=\, \omega \sigma_{\epsilon}(\omega)\, . 
\end{equation}
We define $B^{\tau, \id}$ to be the union of all maximal components.

Consider the non-abelian group cohomology associated to the action of $\ZZ/2\ZZ$ on $W$
given by $\sigma_{\epsilon}$. Note that the cocycle condition is
precisely \eqref{eq cocycle condition}, that is, the condition for $B^\tau_\omega$ to
be maximal. Also, the coboundary condition is precisely
\eqref{eq coboundary condition}. These tell us that maximal components
are indexed by $H^1(\sigma_{\epsilon},\, W)$,
\begin{equation} \label{eq decomposition of B^tau_max}
(A \otimes_\ZZ \Lambda_T/W)^{\tau, \id} \,= \,\bigcup_{\overline{\omega} \in H^1(\sigma_{\epsilon},W)} (A \otimes_\ZZ \Lambda_T/W)^\tau_{\overline{\omega}}.
\end{equation}

Analogously, for a fixed representative of each $\gamma \,\in \,\Upsilon_{\sigma_{\epsilon}}$ one
can define a $\gamma$-shifted cocycle condition
\begin{equation} \label{eq shifted cocycle condition}
\gamma \,= \,\omega \sigma_{\epsilon}(\omega)\, . 
\end{equation}
We also consider the union of all the components with a fixed $\gamma
\,\in\, \Upsilon_{\sigma_{\epsilon}}$,
\[
(A \otimes_\ZZ \Lambda_T/W)^{\tau,\gamma} \,:=\, \bigcup_{\delta([\omega]) = \gamma} (A \otimes_\ZZ \Lambda_T/W)^\tau_\omega\, ;
\]
it is clear that
\[
(A \otimes_\ZZ \Lambda_T/W)^\tau \,= \,
\bigcup_{\gamma \in \Upsilon_{\sigma_{\epsilon}}} (A \otimes_\ZZ \Lambda_T/W)^{\tau,\gamma}\, .
\]

Taking the coboundary condition \eqref{eq coboundary condition} as before, we can define the shifted non-abelian group cohomology $H^1_\gamma(\sigma_{\epsilon},\,W)$. We see that 
\begin{equation} \label{eq decomposition of B^tau_gamma}
(A \otimes_\ZZ \Lambda_T/W)^{\tau, \gamma} \,=\,
 \bigcup_{\overline{\omega} \in H_\gamma^1(\sigma_{\epsilon},W)} (A \otimes_\ZZ \Lambda_T/W)^\tau_{\overline{\omega}}\, .
\end{equation}

One can easily see that every point $b \,\in\, \dot{B}$ such that
\[
b \,=\, \omega \dot{\tau} (b)
\]
actually satisfies the condition 
\[
b \,=\, \omega \dot{\tau} \omega \dot{\tau} (b)\, .
\]
Recalling \eqref{eq omega tau^2}, we observe that $b$ is fixed by
$\omega \sigma_{\epsilon}(\omega)$, and this implies that
\begin{equation} \label{eq non maximal contained in sing}
\dot{B}^{\omega \dot{\tau}} \,\subset\, \dot{B}^{\omega \sigma_{\epsilon}(\omega)}\, .
\end{equation}

\begin{remark}
Note that if $\dim(A) \,\geq\, 2$, then the singular locus of $B = \dot{B}/W$ is
\[
\Sing(B) \,=\, \bigcup_{\id \neq \omega \in W} p(\dot{B}^{\omega})\, .
\]
An antiholomorphic involution on a complex manifold has a half
dimensional fixed point locus. In our case, non-maximal components
$B^\tau_\omega \,\in\, B^{\tau,\gamma}$ have dimension less than $\frac{1}{2}
\dim(B)$. Although, \eqref{eq non maximal contained in sing} implies that any non-maximal
component satisfies $B^\tau_\omega \,\subset\, B^{\tau,\gamma}$, and therefore we have
$B^\tau_\omega \,\subset\,\Sing(B)$. 
\end{remark}

\subsection{Moduli spaces of pseudo-real Higgs bundles}
\label{sc moduli of pseudo-real}

Every element $z \,\in\, Z_2$ of order $2$ of the center defines an element of the Weyl group $\omega_z$ in the following way (see for instance \cite{friedman&morgan, FGN}). Take an alcove $A \subset \tT$ containing the origin. We know (see for instance \cite{broker&tomDieck}) that there is a vertex $a_z$ of the alcove $A$ such that $z = \exp(a_z)$. We see that $A - a_z$ is
another alcove containing the origin. Hence there is a unique element $\omega_z \,\in\,
W$ such that
\[
A - a_z \,=\, \omega_z(A)\, .
\]
In the trivial case we obviously have $\omega_0 \,=\, \id$. 

\begin{remark} \label{rm omega_z = acting by z}
Note that the action of $\omega_z$ on $T$ coincides with the action of $z$ on it.
\end{remark} 

Now we can provide a description of the moduli spaces of pseudo-real $G$-Higgs bundles.

\begin{theorem}
Take the central element $z \in Z^{\sigma_{-}}_2$, the antiholomorphic involution
$$t_y \circ \alpha_{(-,a)} \,: \, X \,\longrightarrow\, X\, ,$$ and the pair of holomorphic and anti-holomorphic involutions $\sigma_+, \sigma_{-} \,:\, G \,\longrightarrow\, G$ such that $\sigma_+ = \sigma_- \sigma_K$. The image of the moduli space of $(G,\alpha_{(-,a)},\sigma_{-},\pm,z)$-Higgs bundles under the forgetful morphism is
\[
\widetilde{\Mm}(G,t_y \circ \alpha_{(-, a)},\sigma_{-},\pm,z)\,\,=\, \bigcup_{\overline{\omega} \in H_{\omega_z}^1(\sigma_{-},W)} \left( T^*\widehat{X} \otimes_\ZZ \Lambda_T \right)^{\omega \left ( \dot{I}(\alpha_{(-, a)},\sigma_{-},\pm) \right)}/N_W(T^{\omega \sigma_{-}})\, .
\]
Furthermore, the forgetful morphism
\[
\Mm(G,t_y \circ \alpha_{(-,a)},\sigma_{-},\pm,z)\, \longrightarrow \widetilde{\Mm}(G,t_y \circ \alpha_{(-,a)},\sigma_{-},\pm,z)\,
\]
is bijective. Analogously, the image of the moduli space of representations is 
\[
\widetilde{\Rr}(G,t_y \circ \alpha_{(-, a)},\sigma_{-},\pm,z)\,\,=\, \bigcup_{\overline{\omega} \in H_{\omega_z}^1(\sigma_{-},W)} \left( \Hom(\pi_1(X), \,\CC^*) \otimes_\ZZ \Lambda_T \right)^{\omega \left ( \dot{J}(\alpha_{(-, a)},\sigma_{-},\pm) \right)}/N_W(T^{\omega \sigma_{-}})\, ,
\]
and the forgetful morphism
\[
\Rr(G,t_y \circ \alpha_{(-,a)},\sigma_{-},\pm,z)\, \longrightarrow \widetilde{\Rr}(G,t_y \circ \alpha_{(-,a)},\sigma_{-},\pm,z)\,
\]
is bijective as well.
\end{theorem} 

\begin{proof}
In view of Remark \ref{rm homogenicity_-}, without any loss of generality,
we can take $\alpha_{(\epsilon, a)}$ to be our anti-holomorphic involution. From
Theorem \ref{th puntos fijos} we have that
$\widetilde{\Mm}(G,\alpha_{(-, a)},\sigma_{-},\pm,z)$ lies in the fixed-point locus
of $I(\alpha_{(-, a)},\sigma_{-},\pm)$. In fact, recalling Remark
\ref{rm omega_z = acting by z} and the equivalent definition of a pseudo-real
$(G,\alpha_{(-,a)},\sigma_{-},\pm,z)$-Higgs bundles given in \eqref{eq theta} and
\eqref{eq theta^2 = z}, we know that the points lying in
$\widetilde{\Mm}(G,\alpha_{(-, a)},\sigma_{-},\pm,z)$ are those lying in the components $(T^*\widehat{X} \otimes_\ZZ \Lambda_T / W)^{\dot{I}(\alpha_{(-, a)},\sigma_{-},\pm)}_{\omega}$ where
\[
\omega \sigma_{-}(\omega) = \omega_c.
\]
From this and \eqref{eq decomposition of B^tau_gamma}, we see that
$\widetilde{\Mm}(G,\alpha_{(-, a)},\sigma_{-},\pm,z)$ is given by the union of the
components of $H_{\omega_z}^1(\sigma_{-},W)$, and the first statement follows from Corollary \ref{co description of the fixed locus of I and J}.

Recall the definition of an isomorphism of pseudo-real Higgs bundles given in 
\ref{eq isomorphic pseudo-real} and the description of the group $N_W(T^{\omega 
\sigma_{-}})$ given in \eqref{eq N_W T}. Then, every isomorphism between pseudo-real 
$G$-Higgs bundles parametrized by $T^*\widehat{X} \otimes_\ZZ \Lambda_T$ is given by the 
elements of $N_W(T^{\omega \sigma_{-}})$. The second statement follows from this 
observation and the description of $\widetilde{\Mm}(G,\alpha_{(-, 
a)},\sigma_{-},\pm,z)$ given above.

Finally, the third an fourth statements follow from the previous description and \cite[Theo\-rems 4.5 and 4.8]{BGP}.
\end{proof}

Recall that for $z \,=\, \id_G$, one has $\omega_z \,=\, \id_W$, and therefore the moduli 
space of real $(G, \alpha_{-}, \sigma_{-}, \pm)$-Higgs bundles (corresponding to $z \,=\, 
\id_G$) is the union of the components classified by $H^1(\sigma_{-}, W)$.

\section*{Tables}

\begin{table}[h]
\begin{center}
    \begin{tabular}{|c|c|c|c|}
    \hline
    Involution & Admissible translations & $X^\alpha$ & $X/\alpha$ \\ \hline
    $\alpha_{(+, 1)}$ & $-$ & $X_\gamma$ & $X_\gamma$
    \\ 
    $t_y \circ \alpha_{(+, 1)}$ & $y \in X[2], y \neq x_0$ & $\emptyset$ & $X_{(\gamma - \widetilde{y})}$
    \\ 
    $\alpha_{(+, -1)}$ & $-$ &  $X[2]$ & $\PP^1$
    \\ 
    $t_y \circ \alpha_{(+, -1)}$ & $y \in X$ &  $\frac{1}{2}y + X[2]$ & $\PP^1$
    \\ \hline
    \end{tabular}
\end{center}
\caption{Holomorphic involutions of $X_\gamma$.}
\label{tb holomorphic involutions}
\end{table}

\begin{table}[h]
\begin{center}
    \begin{tabular}{|c|c|c|c|}
    
    \hline
    
    Involution & $\left ( (\alpha_{\alpha_{\epsilon}})_*\delta_1, (\alpha_{(\epsilon,a}))_*\delta_2 \right )$  \\ \hline
    
    $\alpha_{(+,1)}$ & $(\delta_1,\delta_2)$ \\
    
    $\alpha_{(+,-1)}$ & $(-\delta_1,-\delta_2)$ \\ \hline    
   \end{tabular}
\caption{Action of $\alpha_{(+,a)}$ on $\pi_1(X_\gamma)$.}
\label{tb action of alpha on pi_1 hol}
\end{center}
\end{table}

\begin{table}[h!]
\begin{center}
    \begin{tabular}{|c|c|c|}
    
    \hline
    
     Topological type & $X^{\alpha_-}$ & $X/\alpha_-$ \\ \hline
       $(0,1)$ & $\emptyset$ & Klein bottle \\ 
       $(1,1)$  & $S^1$& M\"obius strip \\ 
       $(2,0)$  & $S^1 \sqcup S^1$ & Closed annulus \\ \hline
   
    \end{tabular}
\caption{Topological type of real elliptic curves.}
\label{zzt}
\end{center}
\end{table} 

\begin{table}[h]
\begin{center}
\begin{tabular}{|c|c|}
\hline
Region of $\HH$ & Values of $\gamma$ 
\\ \hline
$A$ & $\{ \Im(\gamma) > 1, \Re(\gamma) = 0 \}$ 
\\  \hline
$B$ & $\{ \Im(\gamma) = 1, \Re(\gamma) = 0  \}$
\\  \hline
$C$ & $\{ 0 < \Re(\gamma) < \frac{1}{2}, \Im(\gamma) = \sqrt{1 - \Re(\gamma)}  \}$
\\  \hline
$D$ & $\{ \Im(\gamma) = \frac{\sqrt{3}}{2}, \Re(\gamma) = \frac{1}{2}  \}$
\\  \hline
$E$ & $\{ \Im(\gamma) > \frac{\sqrt{3}}{2}, \Re(\gamma) = \frac{1}{2}  \}$
\\  \hline
\end{tabular}
\end{center}
\caption{$X_\gamma$ admitting an antiholomorphic involution.}
\label{tb region of partial Omega}
\end{table} 

\begin{table}[h]
\begin{center}
    \begin{tabular}{|c|c|c|c|c|}
    \hline
    Region of $\HH$ & Involution &  Admissible translations & Topological type \\ \hline
    \multirow{2}{0cm}{$A$} & $\alpha_{(-,\pm 1)}$ & $-$ & (2,0)
    \\
    & $t_y \circ \alpha_{(-,\pm 1)}$ & $y \neq x_0, y \in X^{\alpha_{(-, \mp 1)}} \cong S^1 \sqcup S^1$  & (0,1) 
    \\ \hline
    \multirow{3}{0cm}{$B$} & $\alpha_{(-,\pm 1)}$ & $-$ & (2,0)
    \\    
    & $\alpha_{(-,\pm \i)}$ & $-$ & (1,1)
    \\
    & $t_y \circ \alpha_{(-,\pm 1)}$ &  $y \neq x_0, y \in X^{\alpha_{(-, \mp 1)}} \cong S^1 \sqcup S^1$ & (0,1) 
    \\ \hline
    \multirow{1}{0cm}{$C$} & $t_y \circ \alpha_{(-,\pm \gamma)}$ & $y \in X^{\alpha_{(-, \mp \gamma)}} \cong S^1$ & (1,1)
    \\ \hline
    \multirow{3}{0cm}{$D$} & $t_y \circ \alpha_{(-,\pm1)}$ & $y \in X^{\alpha_{(-, \mp 1)}} \cong S^1$ & (1,1)
    \\ 
    & $t_y \circ \alpha_{(-,\pm\gamma)}$ & $y \in X^{\alpha_{(-, \mp \gamma)}} \cong S^1$  & (1,1) 
    \\    
    & $t_y \circ \alpha_{(-,\pm\gamma^2)}$ & $y \in X^{\alpha_{(-, \mp \gamma^2)}} \cong S^1$ & (1,1)
    \\ \hline
    \multirow{1}{0cm}{$E$} & $t_y \circ \alpha_{(-,\pm 1)}$  & $y \in X^{\alpha_{(-, \mp 1)}} \cong S^1$ & (1,1)
    \\ \hline
    \end{tabular}
    \caption{Anti-holomorphic involutions on $X_\gamma$.}
    \label{tb all the involutions}
\end{center}
\end{table}

\newpage

\begin{table}[h]
\begin{center}
    \begin{tabular}{|c|c|c|c|}
    
    \hline
    Region of $\HH$ & Involution & $(\alpha_{(\epsilon,a)})_*(\delta_1, \delta_2)$ 
    \\ \hline
    \multirow{2}{1cm}{$A,B$} & $\alpha_{(-,1)}$ & $(\delta_1,-\delta_2)$
    \\ 
    & $\alpha_{(-,-1)}$ & $(-\delta_1,\delta_2)$ 
    \\ \hline
    \multirow{4}{1cm}{$C,D$} & $\alpha_{(-,1)}$ & $(\delta_1,-\delta_2)$
    \\
    & $\alpha_{(-,-1)}$ & $(-\delta_1,\delta_2)$ 
    \\ 
    & $\alpha_{(-,\gamma)}$ & $(\delta_2,\delta_1)$
    \\
    & $\alpha_{(-,-\gamma)}$ & $(-\delta_2,-\delta_1)$ 
    \\ \hline
    \multirow{2}{0cm}{$E$} & $\alpha_{(-,1)}$ & $(\delta_1,-\delta_2+\delta_1)$
    \\ 
    & $\alpha_{(-,-1)}$ & $(-\delta_1,\delta_2-\delta_1)$ 
    \\ \hline
    \end{tabular}
\end{center}
\caption{Action of $\alpha_{(-,a)}$ on $\pi_1(X)$.}
\label{tb action of alpha on pi_1 antihol}
\end{table}

\begin{table}[h]
\begin{center}
    \begin{tabular}{|c|c|c|c|c|}
    
    \hline
    Region of $\HH$ & Involution & $f^+_{(\epsilon,a,R)}(z_1,z_2)$ & $f^-_{(\epsilon,a,R)}(z_1,z_2)$
    \\ \hline
    \multirow{2}{0.5cm}{$\HH$} & $\alpha_{(+,1)}$  & $(z_1,z_2)$ & $(\overline{z}_1^{-1},\overline{z}_2^{-1})$
    \\ 
    & $\alpha_{(+,-1)}$  & $(z_1^{-1},z_2^{-1})$ & $(\overline{z}_1,\overline{z}_2)$  
    \\ \hline
     \multirow{2}{1cm}{$A,B$} & $\alpha_{(-,1)}$ & $(z_1,z_2^{-1})$ & $(\overline{z}_1^{-1},\overline{z}_2)$ 
    \\ 
    & $\alpha_{(-,-1)}$ & $(z_1^{-1},z_2)$ & $(\overline{z}_1,\overline{z}_2^{-1})$
    \\ \hline
    \multirow{4}{1cm}{$C,D$} & $\alpha_{(-,1)}$ & $(z_1,z_2^{-1})$ & $(\overline{z}_1^{-1},\overline{z}_2)$
    \\
    & $\alpha_{(-,-1)}$ & $(z_1^{-1},z_2)$ & $(\overline{z}_1,\overline{z}_2^{-1})$
    \\ 
    & $\alpha_{(-,\gamma)}$ & $(z_2,z_1)$ & $(\overline{z}_2^{-1},\overline{z}_1^{-1})$
    \\
    & $\alpha_{(-,-\gamma)}$ & $(z_2^{-1},z_1^{-1})$ & $(\overline{z}_2,\overline{z}_1)$
    \\ \hline
    \multirow{2}{0.5cm}{$E$} & $\alpha_{(-,1)}$ & $(z_1,z_2^{-1}z_1)$ & $(\overline{z}_1^{-1},\overline{z}_2\overline{z}_1^{-1})$
    \\ 
    & $\alpha_{(-,-1)}$& $(z_1^{-1},z_2z_1^{-1})$ & $(\overline{z}_1^{-1},\overline{z}_2^{-1}\overline{z}_1)$  
    \\ \hline
    \end{tabular}
\end{center}
\caption{Values of $f^{\pm}_{(\epsilon,a,R)}$.}
\label{tb values of f^pm}
\end{table}

\end{document}